\documentclass[11pt]{amsart}
\usepackage{amsmath,amssymb,amsthm,mathrsfs}
\usepackage[all]{xy}

\newtheorem{Def}{Definition}[section]
\newtheorem{Thm}[Def]{Theorem}
\newtheorem{Prop}[Def]{Proposition}
\newtheorem{Cor}[Def]{Corollary}
\newtheorem{Lem}[Def]{Lemma}
\newtheorem{Ex}[Def]{Example}
\newtheorem{Conj}[Def]{Conjecture}

\newcommand{\C}{\mathbb{C}}

\newcommand{\Z}{\mathbb{Z}}
\newcommand{\N}{\mathbb{N}}

\newcommand{\GL}{\mathop{\mathrm{GL}}\nolimits}

\newcommand{\rk}{\mathop{\mathrm{rank}}\nolimits}

\newcommand{\StrO}{\mathop{\mathscr{O}}\nolimits}

\begin{document}
\title[Pfaffian Calabi--Yau Threefolds and Mirror Symmetry]
{Pfaffian Calabi--Yau Threefolds \\ and Mirror Symmetry}
\author{Atsushi Kanazawa}
\date{}

\maketitle

\begin{abstract}
The aim of this article is to report on recent progress in understanding mirror symmetry for some non-complete intersection Calabi--Yau threefolds.
We first construct four new smooth non-complete intersection Calabi--Yau threefolds with $h^{1,1}=1$, 
whose existence was previously conjectured by C. van Enckevort and D. van Straten in \cite{van2}. 
We then compute the period integrals of candidate mirror families of F. Tonoli's degree 13 Calabi--Yau threefold and three of the new Calabi--Yau threefolds. 
The Picard--Fuchs equations coincide with the expected Calabi--Yau equations listed in \cite{van1,van2}. 
Some of the mirror families turn out to have two maximally unipotent monodromy points. 
\end{abstract}

\section{Introduction}
The aim of this article is to report on recent progress in understanding mirror symmetry for some non-complete intersection Calabi--Yau threefolds.
Throughout this paper we adopt the following definition. 
\begin{Def}
A $d$-dimensional Calabi--Yau variety $X$ is a normal compact variety over $\C$ with at worst Gorenstein canonical singularities 
and with trivial dualizing sheaf $\omega_{X} \cong \mathscr{O}_{X}$ such that $H^{i}(X, \mathscr{O}_{X})=0, \ ( i = 1, \dots d-1)$. 
\end{Def}
Among smooth Calabi--Yau threefolds, those with 1-dimensional K\"{a}hler moduli spaces have been attracting much attention 
because their expected mirror partners have 1-dimensional complex moduli spaces and hence one can work on them in detail. 
There are around thirty known examples of topologically distinct smooth Calabi--Yau threefolds with $h^{1,1}=1$, 
most of which are complete intersections of hypersurfaces in toric varieties or homogeneous spaces. 
Although non-complete intersection Calabi--Yau threefolds are only partially explored, 
they are intriguing on their own and provide important testing grounds for mirror symmetry. 
We hope that new non-complete intersection Calabi--Yau threefolds with $h^{1,1}=1$ and mirror phenomena we report in this paper are of interest and will be the first step toward the future investigations. 
This paper is clearly influenced by E. R\o dland's work \cite{rod} and we owe a lot of arguments to it. 
We mention it here and do not repeat it each time in the sequel.\\

In the following we give a brief overview of this paper. 
Section 2 is mainly devoted to the study of pfaffian threefolds in weighted projective spaces. 
Pfaffian Calabi--Yau threefolds in $\mathbb{P}^{6}$ was first studied by F. Tonoli in his thesis \cite{to}.  
By replacing the ambient space $\mathbb{P}^{6}$ by weighted projective spaces, 
we obtain several new low degree Calabi--Yau threefolds with $h^{1,1}=1$.   
We then determine their fundamental topological invariants $\int_{X}H^{3}$, $\int_{X}c_{2}(X)\cdot H$ and $\int_{X}c_{3}(X)$, 
which determine the diffeomorphism class of $X$ when $X$ is simply connected and $h^{1,1}=1$ (Wall's classification theorem \cite{Wall}).  
The main result of Section 2 is the following.   
\begin{Thm}
There exist pfaffian threefolds $X_{5}$, $X_{7}$, $X_{10}$ and $X_{25}$, which are smooth and Calabi--Yau with the following topological invariants.
\begin{center}
 \begin{tabular}{|c|c|c|c|c|c|c|}  \hline
 $X_{i}$  & $h^{1,1}$ & $h^{1,2}$ & $\int_{X_{i}}H^{3}$ & $\int_{X_{i}}c_{2}(X_{i})\cdot H$ \\ \hline
 $X_{5}$ & $1$ & $51$ & $5$  &  $38$  \\ \hline
 $X_{7}$ &$1$ & $61$ & $7$ &  $46$  \\ \hline
$X_{10}$ &  $1$ & $59$ & $10$ & $52$\\ \hline
$X_{25}$ &  $1$ & $51$ & $25$ & $70$\\ \hline
 \end{tabular}
 \end{center}
\end{Thm}
The existence of Calabi--Yau threefolds with these topological invariants was previously conjectured by C. van Enckevort and D. van Straten 
based on the classification of Calabi--Yau equations in \cite{van1,van2}. \\

In Section 3 and 4, we report on mirror symmetry for these Calabi--Yau threefolds. 
A pfaffian Calabi--Yau threefold $X_{13}$ of degree $13$ was constructed by F. Tonoli \cite{to} 
and later a candidate mirror family of $X_{13}$ was proposed by J. B\"{o}hm from the viewpoint of tropical geometry \cite{boehm}. 
We confirm the proposal by computing the Picard--Fuchs equation of the family. 
After computing the conjectural genus $g=0,1$ BPS invariants $n_{d}^{g} \ (d\in \N)$,  
we heuristically determine the number of degree 1 rational curves in $X_{13}$ and find that it coincides with $n_{1}^{0}$ as mirror symmetry predicts. 
Interestingly, the mirror family has a special point where all the indices of the Picard--Fuchs operator are $1/2$, 
in addition to the usual maximally unipotent monodromy point at $0 \in \mathbb{P}^1$. 
This observation is further discussed in comparison with E. R\o dland's work \cite{rod}. \\

Although the existence of mirror family of a given Calabi--Yau threefold is highly non-trivial, 
inspired by the mirror family of $X_{13}$, we exhibit explicit mirror families of the Calabi--Yau threefolds $X_{5}$, $X_{7}$ and $X_{10}$. 
We verify that their Picard--Fuchs equations coincide with the expected Calabi--Yau equations listed in \cite{van1,van2}.  
A general member of these families is quite singular and it has not been settled yet whether a general member of the families admits any crepant resolution or not. \\

Section 5 studies a degree 9 pfaffian Calabi--Yau threefold $X_{9} \subset \mathbb{P}(1^{6},2)$, 
which is isomorphic to a complete intersection Calabi--Yau threefold $\mathbb{P}^{5}_{3^{2}}$. 
This twofold interpretation yields non-isomorphic special one-parameter families, both of which have the same Picard-Fuchs equation. 
These two families may bridge our pfaffian mirror construction and the conventional Batyrev--Borisov mirror construction.\\

It is worth mentioning a relevant work; 
based on the results of this paper, physicists M. Shimizu and H. Suzuki studied open mirror symmetry for our pfaffian Calabi--Yau threefolds \cite{SS}.

\subsection*{Aknowledgement}
The author would like to express his gratitude to Shinobu Hosono for his dedicated support and continuous encouragement. 
He greatly appreciates many helpful discussions with Makoto Miura at various stage of this work. 
His thanks also go to Duco van Straten for letting him know the new data base \cite{van1} and comments on the preliminary version of this paper. 

\section{Pfaffian Calabi--Yau Threefolds}

\subsection{Pfaffian Threefolds in Projective Spaces}
Suppose that $R$ is a regular local ring and $I\subset R$ is an ideal of height 1 or 2. 
J. -P. Serre proved that $R/I$ is Gorenstein if and only if it is complete intersection. 
This is no longer true for height 3 ideals, but such Gorenstein ideals are characterized as pfaffian ideals of certain skew-symmetric matrices \cite{be}. 
This observation suggests that pfaffian varieties form a reasonable class of varieties to study when we investigate non-complete intersection Gorenstein varieties.  
In this subsection we review the basics of pfaffians in order to make this paper self-contained. \\

Throughout this paper we work over complex numbers $\C$. 
Let SkewSym$(n,\C)$ be the set of  $n\times n$ skew symmetric matrices. 
For $N=(n_{i,j}) \in $ SkewSym$(n,\C)$ the pfaffian Pf$(N)$ is defined as
$$
 \mathrm{Pf}(N)= \frac{1}{r!2^{r}}\sum_{\sigma \in \mathfrak{S}_{2r}}\mathrm{sign}(\sigma)\prod_{i=1}^{r} n_{\sigma(i)\sigma(r+i)}
$$
 if $n=2r$ is even, and Pf$(N)=0$ if $n$ is odd. 
 It can be check that Pf$(N)^{2}=\mathrm{det}(N)$.
 We define $N_{i_{1},\dots,i_{l}}$ as a skew-symmetric matrix obtained by removing all the $i_{j}$-th rows and columns from $N$,
 and set $P_{i_{1},\dots,i_{l}}=\mathrm{Pf}(N_{i_{1},\dots,i_{l}})$. 
 Let us next assume that $n$ is odd, say $n=2r+1$.   
 The adjoint matrix adj$(N)$ of $N$ is the rank 1 matrix given by 
$$
 \mathrm{adj}(N)=P\cdot P^{t}, \ \ P=(P_{1},P_{2},\dots,P_{2r+1})^{t}.
$$
We then have $P\cdot N=\det(N)\cdot E_{n}=0$ and if rank$(N)=2r$ then $P_{1},\dots,P_{2r+1}$ generate $Ker(N)$. 
$\GL(n,\C)$ acts on SkewSym$(n,\C)$ by conjugation with a finite number of orbits $\{O_{2i}\}_{i=0}^{r}$ 
where the orbit $O_{2i}$ consists of all skew-symmetric matrices of rank $2i$. 
The closure $\overline{O_{2i}}$ of $O_{2i}$ is singular along its boundary $\overline{O_{2i}}\setminus O_{2i}$ which consists of the union of $\{O_{2j}\}_{j=0}^{i-1}$. \\

Let us first recall the construction of pfaffian varieties in $\mathbb{P}^{n} \ (n>3)$. 
From now on we shall always identify $H^{2}(\mathbb{P}^{n},\Z)\cong \Z$ via the hyperplane class. 
Given an integer $t \in \mathbb{Z}$ and a locally free sheaf $\mathscr{E}$ of odd rank $2r+1$ on $\mathbb{P}^{n}$,  
a global section $N \in H^{0}(\mathbb{P}^{n}, \wedge^{2}\mathscr{E}(t))$ defines an alternating morphism
$\mathscr{E}^{\vee}(-t)\stackrel{N}{\rightarrow} \mathscr{E}$.
The pfaffian complex associated to $(t,\mathscr{E},N)$ is given by 
$$
0 \longrightarrow \mathscr{O}_{\mathbb{P}^{n}}(-t-2s) \stackrel{P^{t}}{\longrightarrow}
\mathscr{E}^{\vee}(-t-s)\stackrel{N}{\longrightarrow}
\mathscr{E}(-s) \stackrel{P}{\longrightarrow} \mathscr{O}_{\mathbb{P}^{n}},
$$
where $s=c_{1}(\mathscr{E})+rt$ and $P$ is defined as
$$
P = \frac{1}{r!}\wedge^{r} N \in H^{0}(\mathbb{P}^{n},\wedge^{2r}\mathscr{E}(rt)).
$$
The first and third morphisms are given by taking the wedge product with $P$ and $P^{t}$ respectively.  
Once we fix a basis of sections $e_{1},\dots, e_{2r+1}$ of $\mathscr{E}$, $N$ may be expressed as a matrix and then $P$ is given by  
$$P=\sum_{i=1}^{2r+1}\mathrm{Pf}(N_{i})\bigwedge_{j\neq i}e_{j}.$$
\begin{Def}
A projective variety $X \subset \mathbb{P}^{n}$ is called the pfaffian variety associated to $(t,\mathscr{E},N)$ 
if the structure sheaf $\mathscr{O}_{X}$ is given by Coker$(P)$. 
The sheaf Im$(P)$ $\subset \mathscr{O}_{\mathbb{P}^{n}}$ is called the pfaffian ideal sheaf of $X$ and denoted by $\mathscr{I}_{X}$. 
\end{Def}

The twist $t$ is usually fixed and often omitted without harm. 
In case a choice of global section $N$ is not explicitly specified, it is understood that $N$ is a general element of $H^{0}(\mathbb{P}^{n}, \wedge^{2}\mathscr{E}(t))$. 
In his thesis \cite{to}, F. Tonoli constructed several smooth Calabi--Yau threefolds with $h^{1,1}=1$ in $\mathbb{P}^{6}$, 
using the following (globalized version of the classical) theorem of D. A. Buchsbaum and D. Eisenbud. 

\begin{Thm}[D. A. Buchsbaum and D. Eisenbud \cite{be}]
Let $X \subset \mathbb{P}^{n}$ be a pfaffian variety associated to $(t,\mathscr{E},N)$. 
$X$ is then the degeneracy locus of the skew-symmetric map $N$ and if $N$ is generically of rank 2r it degenerates to rank 2r-2 in the expected codimension 3, 
in which case, the pfaffian complex gives the self-dual resolution of the ideal sheaf of $X$. 
Moreover, $X$ is locally Gorenstein, subcanonical with $\omega_{X} \cong \mathscr{O}_{X}(t+2s-n-1)$.
\end{Thm}

Let $X$ be a pfaffian Calabi--Yau variety of dimension 3 in $\mathbb{P}^{6}$. 
Here we have $t+2s=7$. 
By applying suitable twists, we can assume that $s=3$ and henceforth we consider the pfaffian complex of the following type.
$$
0 \longrightarrow \mathscr{O}_{\mathbb{P}^{6}}(-7) \stackrel{P^{t}}{\longrightarrow}
\mathscr{E}^{\vee}(-4)\stackrel{N}{\longrightarrow}
\mathscr{E}(-3) \stackrel{P}{\longrightarrow} \mathscr{O}_{\mathbb{P}^{n}}
\longrightarrow \mathscr{O}_{X} \longrightarrow 0
$$

It is natural to expect some bounding of topological invariants of pfaffian Calabi--Yau threefolds in $\mathbb{P}^{6}$. 
To see the range of possible degree, we reduce the question to the compact complex surface theory by taking a hyperplane section. 
Let $S$ be a compact, smooth complex surface. 
There are two important numerical invariants of $S$, namely the geometric genus $p_{g}(S)=\dim H^{0}(S,K_{S})$ and the self-intersection of the canonical divisor $K_{S}^{2}$. 

\begin{Thm}[Castelnuovo inequality] Let $S$ be a minimal surface of general type.
 If the canonical map $\Phi_{|K_{S}|}:S \rightarrow \mathbb{P}^{n}$ is birational to the image, then $K_{S}^{2} \ge 3p_{g}(S)-7$.
\end{Thm}
We say an embedded variety $X\subset \mathbb{P}^{n}$ is full if $X$ is not contained in any hyperplane $\mathbb{P}^{n-1}\subset \mathbb{P}^{n}$.  
Let $S$ be a smooth surface obtained by taking a hyperplane section of a full Calabi--Yau threefold $X \subset \mathbb{P}^{6}$. 
Then deg$(X)=K_{S}^{2}$ and $K_{S}$ is nef since $S$ is a canonical surface. 
As $X \subset \mathbb{P}^{6}$ is full, the short exact sequence $0 \longrightarrow \mathscr{O}_{X} \longrightarrow \mathscr{O}_{X}(1) \longrightarrow \mathscr{O}_{S}(1) \longrightarrow 0 $ 
yields $p_{g}(S)=6$. 
We thereby conclude that the lower bound of the degree of $X$ is $11$. 
Since a complete intersection Calabi--Yau threefold $\mathbb{P}^{5}_{3^{2}}$ has degree 9, we cannot remove the fullness on $X$.\\

In his paper \cite{to}, F. Tonoli constructed pfaffian Calabi--Yau threefolds of degree $d$ in the range $11\le d \le 17$. 
Although the Castelnuovo inequality tells us that the minimal degree $d$ of a full Calabi--Yau threefold in $\mathbb{P}^{6}$ is 11, 
there seems no smooth degree 11 Calabi--Yau threefold in $\mathbb{P}^{6}$. 
Degree 12 pfaffian Calabi--Yau threefolds are complete intersections $\mathbb{P}^{6}_{2^{2},3}$. 
Therefore degree 13 is a good starting point to analyze. 

\begin{Def}[F. Tonoli \cite{to}]We define $X_{13} \subset \mathbb{P}^{6}$ as a pfaffian threefold associated to the locally free sheaf 
$\mathscr{E}=\mathscr{O}_{\mathbb{P}^{6}}(1) \oplus \mathscr{O}_{\mathbb{P}^{6}}^{\oplus 4}$. 
 \end{Def}
$X_{13}$ is indeed a smooth Calabi--Yau threefold and the geometric invariants are given by 
$$
h^{1,1}=1, \ h^{1,2}=61, \ \int_{X_{13}}H^{3}=13, \ \int_{X_{13}}c_{2}(X_{13})\cdot H = 58. 
$$

A degree 14 pfaffian Calabi--Yau threefold $X_{14}$ is defined as a pfaffian threefold associated to the locally free sheaf $\mathscr{E}=\mathscr{O}_{\mathbb{P}^{6}}^{\oplus7}$. 
This is nothing but the intersection of $\mathbb{P}^{6}$ with Pfaff$(7)\subset \mathbb{P}^{20}$, the rank 4 locus of projectivised general skew-symmetric $7 \times 7 $ matrices
$$
\mathbb{P}  (\bigwedge^{2} \mathbb{C}^{7}) = \mathbb{P}(\mathrm{SkewSym(7,\C)}) 
\supset \mathrm{Pfaff(7)}=\{[M] \ | \ \mathrm{rank}(M)\le4 \}. 
$$
It is verified that $X_{14}$ and its mirror partner $\check{X}_{14}$ have rich mathematical structures in \cite{rod, bc,hos}.

\subsection{Pfaffian Threefolds in Weighted Projective Spaces}
F. Tonoli's construction may be generalized by replacing the ambient space $\mathbb{P}^{6}$ by any Fano variety. 
Special care must be taken when the ambient space is singular. 
In the following we shall study the simplest case, when the ambient space is a weighted projective space $\mathbb{P}_{\mathbf{w}}$. 
Given an integer $t \in \mathbb{Z}$ and a locally free sheaf $\mathscr{E}$ of odd rank $2r+1$ on $\mathbb{P}_{\mathbf{w}}$, 
a global section $N \in H^{0}(\mathbb{P}_{\mathbf{w}}, \wedge^{2}\mathscr{E}(t))$ defines an alternating morphism
$\mathscr{E}^{\vee}(-t)\stackrel{N}{\rightarrow} \mathscr{E}$. 
The pfaffian complex associated to $(t,\mathscr{E},N)$ is given by 
\begin{equation}
0 \longrightarrow \mathscr{O}_{\mathbb{P}_{\mathbf{w}}}(-t-2s) \stackrel{P^{t}}{\longrightarrow}
\mathscr{E}^{\vee}(-t-s)\stackrel{N}{\longrightarrow}
\mathscr{E}(-s) \stackrel{P}{\longrightarrow} \mathscr{O}_{\mathbb{P}_{\mathbf{w}}}\notag,
\end{equation}
where $s=c_{1}(\mathscr{E})+rt$ and $P=\frac{1}{r!}\wedge^{r} N $ as before.  
The pfaffian variety $X \subset \mathbb{P}_{\mathbf{w}}$ associated to $(t,\mathscr{E},N)$ is a variety
whose structure sheaf $\mathscr{O}_{X}$ is given by $Coker(P)$. 
We define $|\mathbf{w}|$ as a sum of weights of $\mathbb{P}_{\mathbf{w}}$. 

\begin{Prop}
Let $\mathbb{P}_{\mathbf{w}}$ be a weighted projective space of dimension $6$ and $(t,\mathscr{E},N)$ as above. 
The pfaffian threefold $X$ associate to $(t,\mathscr{E},N)$ has trivial dualizing sheaf $\omega_{X} \cong \mathscr{O}_{X}$ if and only if $t+2s=|\mathbf{w}|$. 
\end{Prop}
\begin{proof} 
Apply the functor $\mathscr{H}om(-, \omega_{\mathbb{P}_{\mathbf{w}}})$ to the pfaffian resolution to compute 
the dualizing sheaf $\omega_{X} \cong \mathscr{E}xt^{3}(\mathscr{O}_{X},\omega_{\mathbb{P}_{\mathbf{w}}})$, 
which is isomorphic to $\cong \mathscr{O}_{X}$ if and only if $t+2s=|\mathbf{w}|$ by the definition of pfaffian variety. 
 \end{proof} 

As we are interested in Calabi--Yau threefolds, we restrict ourselves to the case $t+2s=|\mathbf{w}|$. 
Moreover, up to an opportune twist, we may assume $t=1$ for $|\mathbf{w}|$ odd and $t=0$ for $|\mathbf{w}|$ even. 

\begin{Def}
Define $X_{i}$ as a pfaffian threefold associated to the following locally sheaf $\mathscr{E}_{i}$ on $\mathbb{P}_{\mathbf{w}_{i}}$ for $i=5,7,10$. 
\begin{center}
 \begin{tabular}{|c|c|c|}  \hline
 $i$ & $\mathbf{w}_{i}$ & $\mathscr{E}_{i}$\\ \hline
 $5$ & $(1^{4},2^{3})$ & $\mathscr{O}_{\mathbb{P}_{\mathbf{w}_{5}}}^{\oplus 5}(1)$  \\ \hline
 $7$ & $(1^{5},2^{2})$ & $\mathscr{O}_{\mathbb{P}_{\mathbf{w}_{7}}}(1)^{\oplus 2}\oplus \mathscr{O}_{\mathbb{P}_{\mathbf{w}_{7}}}^{\oplus 3}$ \\ \hline
 $10$ & $(1^{6},2^{1})$ & $\mathscr{O}_{\mathbb{P}_{\mathbf{w}_{10}}}(1)^{\oplus 4}\oplus \mathscr{O}_{\mathbb{P}_{\mathbf{w}_{10}}}$\\ \hline
 \end{tabular}
 \end{center}
 \end{Def}

There are many other choices of weights $\mathbf{w}$ and locally sheaves $\mathscr{E}$ on $\mathbb{P}_{\mathbf{w}}$ to produce pfaffian Calabi--Yau threefolds 
but it seems only three cases above yield smooth Calabi--Yau threefolds in weighted projective spaces of dimension $6$. 

\begin{Thm}
For a generic choice of $N\in H^{0}(\mathbb{P}_{\mathbf{w}_{i}}, \wedge^{2}\mathscr{E}_{i}(t))$, the pfaffian varieties $X_{5}$, $X_{7}$ and $X_{10}$ are smooth varieties.
\end{Thm}
\begin{proof}
A generic choice of $N$ guarantees quasi-smoothness of $X_{i}$ as follows. 
For $X_{5}$ we have Sing$(\mathbb{P}_{\mathbf{w}_{5}})\cong \mathbb{P}^{2}$ and $X_{5} \cap \mathrm{Sing}(\mathbb{P}_{\mathbf{w}_{5}})$ 
is identified with the intersection of $\mathbb{P}^{2}$ with the rank 2 locus of projectivised general skew-symmetric 5 $\times$ 5 matrices Pfaff$(5)\cap\mathbb{P}^{2}$, which is empty.
For $X_{7}$ the matrix $N$ on Sing$(\mathbb{P}_{\mathbf{w}_{7}}) \cong \mathbb{P}^{1}$ has the following form
$$
N  = \begin{pmatrix}
                0           &  0    & g_{1}     & g_{2}    & g_{3} \\
                0            & 0           & g_{4}     & g_{5}   & g_{6}\\
                -g_{1}  & -g_{4} & 0              & 0   & 0 \\
                -g_{2}  & -g_{5} & 0   & 0           &   0\\
                -g_{3}  & -g_{6} & 0   &  0  & 0 \\
                \end{pmatrix} , 
$$
where $g_{1},\dots,g_{6}$ are linear polynomials of $x_{5}$ and $x_{6}$. 
It is obvious that this has rank greater than $2$ for a generic choice of $g_{1},\dots,g_{6}$. 
Finally, for general $X_{10}$ we have $P_{5}|_{\mathrm{Sing}(\mathbb{P}_{\mathbf{w}_{10}})}\ne0$ 
while $P_{i}|_{\mathrm{Sing}(\mathbb{P}_{\mathbf{w}_{10}})}=0$ for $1\le i \le 4$. 
This completes the proof of quasi-smoothness. 
We henceforth assume that $X_{i}$ avoids the singular locus Sing$(\mathbb{P}_{\mathbf{w}_{i}})$.\\

We denote by $\mathbb{P}_{\mathbf{w}_{i}}^{sm}$ the smooth open subset $\mathbb{P}_{\mathbf{w}_{i}}\setminus  \mathrm{Sing}(\mathbb{P}_{\mathbf{w}_{i}})$. 
Since $H^{0}(\mathbb{P}_{\mathbf{w}_{i}}^{sm}, \wedge^{2}\mathscr{E}_{i}(t))$ is generated by global sections, we have a surjection
$$
H^{0}(\mathbb{P}_{\mathbf{w}_{i}}^{sm}, \wedge^{2}\mathscr{E}_{i}(t)) \otimes_{\C}\StrO_{\mathbb{P}_{\mathbf{w}_{i}}^{sm}}
\longrightarrow \wedge^{2}\mathscr{E}_{i}(t). 
$$
This map induces a morphism $f$ of $\mathbb{P}_{\mathbf{w}_{i}}^{sm}$-schemes of full rank everywhere
\[\xymatrix{
\mathbb{P}_{\mathbf{w}_{i}}^{sm} \times H^{0}(\mathbb{P}_{\mathbf{w}_{i}}^{sm}, \wedge^{2}\mathscr{E}_{i}(t)) \ar[rr]^{f} \ar[rd]_{\pi_{1}} & &
 E=Spec (Sym(\wedge^{2}\mathscr{E}_{i}(t)))  \ar[dl]^{\pi_{2}} \\
                         & \mathbb{P}_{\mathbf{w}_{i}}^{sm} &   \\
}\]
sending $(x,N)\mapsto N(x)$. 
Let $p:\mathbb{P}_{\mathbf{w}_{i}}^{sm} \times H^{0}(\mathbb{P}_{\mathbf{w}_{i}}^{sm}, \wedge^{2}\mathscr{E}_{i}(t)) 
\rightarrow H^{0}(\mathbb{P}_{\mathbf{w}_{i}}^{sm}, \wedge^{2}\mathscr{E}_{i}(t))$ be the second projection. 
Define $E_{2}$ to be the codimension $3$ variety of $E$ whose fiber over a point $x \in \mathbb{P}_{\mathbf{w}_{i}}^{sm}$ is
$$
O_{2} \subset \mathrm{SkewSym}(5,\C) \cong \pi_{2}^{-1}(x)
$$
Note that $O_{2}$ is independent of the identification $\mathrm{SkewSym}(5,\C) \cong \pi_{2}^{-1}(x)$. 
Then $Y=f^{-1}(E_{2})$ is of codimension $3$ 
and singular along $f^{-1}(\mathrm{Sing}(E_{2}))=\mathbb{P}_{\mathbf{w}_{i}}^{sm} \times \{0\}$.  
$p|_{Y\setminus (\mathrm{Sing}(Y))}$ is dominant and generic smoothness of $p|_{Y\setminus (\mathrm{Sing}(Y))}$ proves that 
for a generic choice of $N \in H^{0}(\mathbb{P}_{\mathbf{w}_{i}}^{sm}, \wedge^{2}\mathscr{E}_{i}(t))$
$$
p|_{Y\setminus (\mathrm{Sing}(Y))}^{-1}(N)=\{(x,N)\ | \ \rk(N(x))=2\}
$$
is smooth and of dimension $3$. 
The quasi-smoothness of $X_{i}$ then shows 
$$
\mathbb{P}_{\mathbf{w}_{i}}^{sm} 
\supset X_{i}=\pi_{2}(p|_{Y\setminus (\mathrm{Sing}(Y))}^{-1}(N)) \cong p|_{Y\setminus (\mathrm{Sing}(Y))}^{-1}(N).
$$ 
We have thus proved the theorem. 
\end{proof}

For each $X_{i}$, vanishing of $H^{j}(X_{i}, \mathscr{O}_{X_{i}})=0$ for $j=1,2$ readily follows from the pfaffian resolution. 
Therefore $X_{5}, X_{7}$ and $X_{10}$ are smooth Calabi--Yau threefolds. 
In the following we assign to each $X_{i}$ a polarization $H$ coming from the hyperplane class of the ambient space $\mathbb{P}_{\mathbf{w}_{i}}$. 

\begin{Lem} \label{Hilb}
The Hilbert series $H_{X_{i}}(t)$ of the pfaffian Calabi--Yau threefold $X_{i}$ is given by the following. 
\begin{align}
H_{X_{5}}(t)=&\frac{1+3t^{2}+t^{4}}{(1-t)^{4}} \ \ \ \ \ 
H_{X_{7}}(t)=\frac{1+t+3t^{2}+t^{3}+t^{4}}{(1-t)^{4}} \notag \\
& H_{X_{10}}(t)=\frac{1+2t+4t^{2}+2t^{3}+t^{4}}{(1-t)^{4}} \notag
\end{align}
\end{Lem}
\begin{proof}
As we already have a resolution of the structure sheaf of $X_{i}$, the claim easily follows from the additivity of Hilbert series and 
the formula $$H_{\mathbb{P}_{\mathbf{w}_{i}}}(\mathscr{O}_{\mathbb{P}_{\mathbf{w}_{i}}}(k))(t)=\frac{t^{k}}{\prod_{i=0}^{n}(1-t^{w_{i}}_{i})}.$$ 
\end{proof}
\begin{Prop}
The degree $\int_{X_{i}}H^{3}$ of the pfaffian Calabi--Yau threefold $X_{i}$ is $i$. 
\end{Prop} 
\begin{proof}
Let $d$ be $3!$ times the leading coefficient of the Hilbert polynomial $P_{X_{i}}(t)$, which is readily available thanks to Lemma \ref{Hilb}.   
Since a pfaffian variety is locally a complete intersection, the triple intersection $\int_{X_{i}}H^{3}$ coincides with $d$. 
\end{proof}

\begin{Prop}$\int_{X_{i}}c_{2}(X_{i})\cdot H$ is given below for $i=5,7,10$.  
\begin{center}
 \begin{tabular}{|c|c|c|c|c|c|c|}  \hline
 $X_{i}$  &  $X_{5}$    &  $X_{7}$ &  $X_{10}$\\ \hline
$\int_{X_{i}}c_{2}(X_{i})\cdot H$      &  $38$  &  $46$ &  $52$  \\ \hline
 \end{tabular}
 \end{center}
 \end{Prop}
\begin{proof}
Since we know that $X_{i}$ is a smooth Calabi--Yau threefold, the Hirzebruch--Riemann--Roch Theorem gives
$$
 \chi(X_{i},\mathscr{O}_{X_{i}}(H))
 =\frac{1}{6}\deg(X_{i})+\frac{1}{12}\int_{X_{i}}c_{2}(X_{i})\cdot H.
$$
 By the Kodaira vanishing theorem, $H^{j}(X_{i},\mathscr{O}_{X_{i}}(H))=0$ except for $j=0$ and hence we have 
$$
 \chi(X_{i},\mathscr{O}_{X_{i}}(H))=\dim H^{0}(X_{i},\mathscr{O}_{X_{i}}(H))=\dim H^{0}(\mathbb{P}_{\mathbf{w}_{i}},\mathscr{O}_{\mathbb{P}_{\mathbf{w}_{i}}}(H))
$$
This determines $\int_{X_{i}}c_{2}(X_{i})\cdot H$.
 \end{proof}

We will also need resolutions of the powers of pfaffian ideals, which are studied, for example, in \cite{boffi}. 
Let $R$ be a commutative ring. 
We consider a free $R$-module $E$ of rank $2r+1$ and a generic alternating map $N : E^{\vee} \rightarrow E$, 
then we have the pfaffian resolution 
$$
0 \longrightarrow 
R \stackrel{P^{t}}{\longrightarrow} E^{\vee}
\stackrel{N}{\longrightarrow} E
\stackrel{P}{\longrightarrow} R
\longrightarrow R/I \longrightarrow 0.
$$
Let $L_{\lambda}E$ be the representation of $GL(E)$ corresponding to a hook Young tableau $\lambda$ (we refer the reader to \cite{boffi} for the precise definition of $L_{\lambda}E$).  
\begin{Lem}[\cite{boffi}]\label{Lem:0} There exists a resolution of $I^{2}$ of the form 
$$
0 \longrightarrow  L_{(2r-1)}E \cong \wedge^{2r-1}E
\stackrel{\vartheta_{3}}{\longrightarrow} L_{(2r,1)}E
\stackrel{\vartheta_{2}}{\longrightarrow} L_{(2r+1,1^{2})}E \cong S^{2} E
\stackrel{\vartheta_{1}}{\longrightarrow} I^{2}
\longrightarrow 0,
$$
where $\vartheta_{1}$ is the second symmetric power of $P$ and $\vartheta_{3}$ and $\vartheta_{2}$ are induced by the map
\begin{equation}
\wedge^{a}E \otimes_{R} S^{b}E \rightarrow \wedge^{a+1}E \otimes_{R} S^{b+1}E, \ \
u\otimes v \mapsto \sum_{i,j+1}^{2r+1}n_{i,j} e_{i}\wedge u \otimes v e_{j}, \notag
\end{equation}
where $N=\sum_{i,j=1}^{2r+1}n_{i,j} e_{i}\otimes e_{j}$ with respect to some fixed basis $e_{1},\dots e_{2r+1}$ for $E$.
\end{Lem}

\begin{Lem}\label{Lem:1}
There exist resolutions of the ideal sheaves 
$\mathscr{I}_{X_{5}}^{2}$, $\mathscr{I}_{X_{7}}^{2}$ and $\mathscr{I}_{X_{10}}^{2}$ 
of the following form. 
 \begin{align}
 0 \longrightarrow 
 \mathscr{O}_{\mathbb{P}_{\mathbf{w}_{5}}}(-12) ^{\oplus 10}
 \stackrel{\vartheta_{3}}{\longrightarrow}  \mathscr{O}_{\mathbb{P}_{\mathbf{w}_{5}}}(-10)^{\oplus 24}
\stackrel{\vartheta_{2}}{\longrightarrow} \mathscr{O}_{\mathbb{P}_{\mathbf{w}_{5}}}(-8)^{\oplus15}
\stackrel{\vartheta_{1}}{\longrightarrow} \mathscr{I}_{X_{5}}^{2} \longrightarrow 0 \notag
 \end{align}
 \begin{align}
 0 \longrightarrow 
 \mathscr{O}_{\mathbb{P}_{\mathbf{w}_{7}}}&(-12) \oplus \mathscr{O}_{\mathbb{P}_{\mathbf{w}_{7}}}(-11)^{\oplus 6}
 \oplus \mathscr{O}_{\mathbb{P}_{\mathbf{w}_{7}}}(-10)^{\oplus 3}\notag \\
 \stackrel{\vartheta_{3}}{\longrightarrow} &
 \mathscr{O}_{\mathbb{P}_{\mathbf{w}_{7}}}(-10)^{\oplus 6} \oplus \mathscr{O}_{\mathbb{P}_{\mathbf{w}_{7}}}(-9)^{\oplus 12} 
 \oplus \mathscr{O}_{\mathbb{P}_{\mathbf{w}_{7}}}(-8)^{\oplus 6} \notag \notag \\
&\stackrel{\vartheta_{2}}{\longrightarrow}
\mathscr{O}_{\mathbb{P}_{\mathbf{w}_{7}}}(-8)^{\oplus 6} \oplus \mathscr{O}_{\mathbb{P}_{\mathbf{w}_{7}}}(-7)^{\oplus 6} 
\oplus \mathscr{O}_{\mathbb{P}_{\mathbf{w}_{7}}}(-6)^{\oplus 3}
\stackrel{\vartheta_{1}}{\longrightarrow} \mathscr{I}_{X_{7}}^{2} \longrightarrow 0 \notag
 \end{align}
 \begin{align}
 0 \longrightarrow 
 \mathscr{O}_{\mathbb{P}_{\mathbf{w}_{10}}}(-10)^{\oplus6}\oplus & \mathscr{O}_{\mathbb{P}_{\mathbf{w}_{10}}}(-9)^{\oplus 4} \notag \\
 \stackrel{\vartheta_{3}}{\longrightarrow}
 \mathscr{O}_{\mathbb{P}_{\mathbf{w}_{10}}}& (-9)^{\oplus 4}\oplus  \mathscr{O}_{\mathbb{P}_{\mathbf{w}_{10}}}(-8)^{\oplus 16} \oplus \mathscr{O}_{\mathbb{P}_{\mathbf{w}_{10}}}(-7)^{\oplus 4} \notag \\ 
\stackrel{\vartheta_{2}}{\longrightarrow} & 
\mathscr{O}_{\mathbb{P}_{\mathbf{w}_{10}}}(-8) \oplus \mathscr{O}_{\mathbb{P}_{\mathbf{w}_{10}}}(-7)^{\oplus 4} 
\oplus \mathscr{O}_{\mathbb{P}_{\mathbf{w}_{10}}}(-6)^{\oplus 10} \notag \\
 & \stackrel{\vartheta_{1}}{\longrightarrow}\mathscr{I}_{X_{10}}^{2} \longrightarrow 0 \notag
 \end{align}
Here each term from left to right is regarded as $5 \times 5$ skew-symmetric, general but top left being zero, and symmetric matrices 
and the morphisms are given by $\vartheta_{3}(X)=NX-(NX)_{1,1}I$, $\vartheta_{2}(X)=XN+(XN)^{t}$, $\vartheta_{1}(X)=P^{t}XP$. 
\end{Lem}
\begin{proof}
Let $F$ be a free $R$-module of rank $5$. 
We may suitably identify $\wedge^{3}F$ with $5 \times 5$ skew-symmetric matrices, $L_{(4,1)}F$ with general but top left being zero matrices, and $S^{2} F$ with symmetric matrices. 
By Lemma \ref{Lem:0} it is straightforward to see that the morphisms $\vartheta_{i}$ are of the forms described in the claim.   
\end{proof}

\begin{Thm}\label{Hodge}
The Hodge numbers $h^{1,1}$ and $h^{1,2}$ of the pfaffian Calabi--Yau threefold $X_{i}$ are given by the following table.
\begin{center}
 \begin{tabular}{|c|c|c|c|c|c|c|}  \hline
 $X_{i}$  &  $X_{5}$    &  $X_{7}$  &  $X_{10}$\\ \hline
$h^{1,1}$      &  $1$  &   $1$ &  $1$  \\ \hline
$h^{1,2}$      &  $51$  &   $61$  &  $59$ \\ \hline
 \end{tabular}
 \end{center}
\end{Thm}
\begin{proof}
In this proof, we simply write $X=X_{i}$ and $\mathbb{P}_{\mathbf{w}}=\mathbb{P}_{\mathbf{w}_{i}}$ for some $i=5,7,10$. 
Twisting the pfaffian resolution of the structure sheaf, 
we know that  $H^{i}(X,\mathscr{O}_{X}(-j)) \cong H^{i+3}(\mathbb{P}_{\mathbf{w}},\mathscr{O}_{\mathbb{P}_{\mathbf{w}}}(-|\mathbf{w}|-j))\ (j=1,2)$, 
which do not vanish only for $i=3$. 
Restricting the weighted analogue of the Euler sequence to $X$, we obtain
$$
0 \longrightarrow \varOmega_{\mathbb{P}_{\mathbf{w}}}\otimes_{\mathscr{O}_{\mathbb{P}_{\mathbf{w}}}} \mathscr{O}_{X}
\longrightarrow \bigoplus_{i=0}^{6} \mathscr{O}_{X}(-w_{i})
\longrightarrow \mathscr{O}_{X} \longrightarrow 0.
$$
Since $X$ is a smooth Calabi--Yau threefold, the long exact sequence induced by the short exact sequence above yields 
$H^{i}(X,\varOmega_{\mathbb{P}_{\mathbf{w}}}\otimes_{\mathscr{O}_{\mathbb{P}_{\mathbf{w}}}} \mathscr{O}_{X})=0 \ (i=0,2)$, 
$H^{1}(X,\varOmega_{\mathbb{P}_{\mathbf{w}}}\otimes_{\mathscr{O}_{\mathbb{P}_{\mathbf{w}}}} \mathscr{O}_{X}) \cong \mathbb{C}$ and the exact sequence 
\begin{align}
0 \longrightarrow H^{2}(X,\mathscr{O}_{X})
& \longrightarrow H^{3}(X,\varOmega_{\mathbb{P}_{\mathbf{w}}}\otimes_{\mathscr{O}_{\mathbb{P}_{\mathbf{w}}}} \mathscr{O}_{X}) \notag \\
\longrightarrow & H^{3}(X,\bigoplus_{i=0}^{6} \mathscr{O}_{X}(-w_{i}))
\longrightarrow H^{3}(X, \mathscr{O}_{X}) \longrightarrow 0.\notag 
\end{align}
 We hence have 
 $h^{3}(X,\varOmega_{\mathbb{P}_{\mathbf{w}}}\otimes_{\mathscr{O}_{\mathbb{P}_{\mathbf{w}}}} \mathscr{O}_{X})
 =h^{3}(X,\bigoplus_{i=0}^{6} \mathscr{O}_{X}(-w_{i}))-1$. 
 From the resolution of $\mathscr{F}_{\bullet}\rightarrow\mathscr{I}_{X}^{2}$ in Lemma \ref{Lem:1} we obtain
$$
h^{4}(\mathbb{P}_{\mathbf{w}},\mathscr{I}_{X}^{2})- h^{5}(\mathbb{P}_{\mathbf{w}},\mathscr{I}_{X}^{2})
=\sum_{i=1}^{3}(-1)^{i+1}h^{6}(\mathbb{P}_{\mathbf{w}},\mathscr{F}_{i})-h^{6}(\mathbb{P}_{\mathbf{w}},\mathscr{I}_{X}^{2}).
$$
The pfaffian resolution gives $H^{4}(\mathbb{P}_{\mathbf{w}},\mathscr{I}_{X}) 
\cong H^{6}(\mathbb{P}_{\mathbf{w}},\mathscr{O}_{\mathbb{P}_{\mathbf{w}}}(-|\mathbf{w}|)) \cong \mathbb{C}$ and 
$H^{i}(X,\mathscr{I}_{X})=0$ (otherwise). 
Since we assume that $X$ is smooth, we have the short exact sequence of sheaves 
$0 \rightarrow \mathscr{I}_{X}^{2} \rightarrow \mathscr{I}_{X} \rightarrow \mathscr{N}_{X/\mathbb{P}_{\mathbf{w}}}^{\vee} \rightarrow 0$. 
The induced long exact sequence gives 
$H^{i}(\mathbb{P}_{\mathbf{w}}, \mathscr{N}_{X/\mathbb{P}_{\mathbf{w}}}^{\vee})=0 \ (0 \le i \le 2)$ and 
$H^{5}(\mathbb{P}_{\mathbf{w}},\mathscr{I}_{X}^{2})=H^{6}(\mathbb{P}_{\mathbf{w}},\mathscr{I}_{X}^{2})=0$. 
Moreover, we also have the short exact sequence
$$
0 \longrightarrow H^{3}(\mathbb{P}_{\mathbf{w}}, \mathscr{N}_{X/\mathbb{P}_{\mathbf{w}}}^{\vee})
\longrightarrow   H^{4}(\mathbb{P}_{\mathbf{w}},\mathscr{I}_{X}^{2}) 
\longrightarrow H^{4}(\mathbb{P}_{\mathbf{w}},\mathscr{I}_{X}) \longrightarrow 0. 
$$
It then follows immediately that 
$$
h^{3}(\mathbb{P}_{\mathbf{w}},\mathscr{N}_{X/\mathbb{P}_{\mathbf{w}}}^{\vee})=h^{3}(X,\mathscr{N}_{X/\mathbb{P}_{\mathbf{w}}}^{\vee})
=\sum_{i=1}^{3}(-1)^{i+1}h^{6}(\mathbb{P}_{\mathbf{w}},\mathscr{F}_{i}).
$$
On the other hand, the conormal exact sequence yields the exact sequence 
$$
0 \longrightarrow H^{2}(X,\varOmega_{X})
\longrightarrow H^{3}(X,\mathscr{N}_{X/\mathbb{P}_{\mathbf{w}}}^{\vee}) 
\longrightarrow H^{3}(X, \varOmega_{\mathbb{P}_{\mathbf{w}}}\otimes_{\mathscr{O}_{\mathbb{P}_{\mathbf{w}}}} \mathscr{O}_{X}) \longrightarrow 0
$$
and $H^{1}(X,\varOmega_{X}) \cong \mathbb{C}$. 
We finally establish the formula 
\begin{align}
h^{2}(X,\varOmega_{X})&=h^{3}(X,\mathscr{N}_{X/\mathbb{P}_{\mathbf{w}}}^{\vee})
-h^{3}(\varOmega_{\mathbb{P}_{\mathbf{w}}}\otimes_{\mathscr{O}_{\mathbb{P}_{\mathbf{w}}}} \mathscr{O}_{X_{10}})\notag \\
&=\sum_{i=1}^{3}(-1)^{i+1}h^{6}(\mathbb{P}_{\mathbf{w}},\mathscr{F}_{i})-h^{3}(X,\bigoplus_{i=0}^{6} \mathscr{O}_{X}(-w_{i})).\notag
\end{align}
Therefore $h^{1,2}$ is determined by the explicit description of $\mathscr{F}_{\bullet}\rightarrow\mathscr{I}_{X}^{2}$ derived in Lemma \ref{Lem:1}. 
\end{proof}
 
The existence of smooth Calabi--Yau threefolds $X_{5}$, $X_{7}$ and $X_{10}$ with the computed topological invariants was previously conjectured 
by C. van Enckevort and D. van Straten from the viewpoint of Calabi--Yau equations in \cite{van2}. 
Regrettably it has not been settled yet whether they are simply connected or not.

 \subsection{Complete Intersection Type}
 In this subsection we study complete intersections of pfaffian varieties and hypersufaces in weighted projective spaces. 
 The main idea is to use pfaffian varieties as codimension 3 analogue of hypersurfaces in the ambient space. 
 \begin{Def}
Set $t=1$ and $\mathscr{E}_{25}=\mathscr{O}^{\oplus 5}_{\mathbb{P}^{9}}$.   
Two generic global sections $N_{1}, N_{2} \in H^{0}(\mathbb{P}^{9}, \wedge^{2}\mathscr{E}_{25}(1))$ define alternating morphisms 
$N_{1},N_{2} : \mathscr{E}_{25}^{\vee}(-1) \rightarrow \mathscr{E}_{25}$. 
Define $X_{25}$ as the common degeneracy loci of $N_{1}$ and $N_{2}$. 
 \end{Def}
 
 Since the pfaffian sixfold associated to the data $(\mathscr{E}_{25},N_{i})$ is isomorphic to Gr$(2,5) \subset \mathbb{P}^{9}$, 
 $X_{25}$ may be seen as a complete intersection of two Grassmannians embedded in two different ways $i_{j}:\mathrm{Gr}(2,5)\hookrightarrow \mathbb{P}^{9} \ ( j=1,2)$. 
 $$
 X_{25}=i_{1}(\mathrm{Gr}(2,5))\cap i_{2}(\mathrm{Gr}(2,5))
$$
  
 \begin{Lem}\label{Res}
Let $X$ be the pfaffian variety associated to $(\mathscr{E}_{25},N_{i})$. 
Then $\mathscr{I}_{X}^{2}$ has the following resolution. 
$$
0 \longrightarrow \mathscr{O}_{\mathbb{P}^{9}}(-6)^{\oplus10}\longrightarrow \mathscr{O}_{\mathbb{P}^{9}}(-5)^{\oplus24}
\longrightarrow \mathscr{O}_{\mathbb{P}^{9}}(-4)^{\oplus15}\longrightarrow \mathscr{I}_{X}^{2}\longrightarrow 0
$$
\end{Lem}
\begin{proof}This may be proved in a similar fashion to Lemma \ref{Lem:1}.
\end{proof}
 
 \begin{Prop}
 $X_{25}$ is a smooth Calabi--Yau threefold with the following topological invariants.
$$
h^{1,1}=1, \ h^{1,2}=51, \ \int_{X_{25}}H^{3}=25, \ \int_{X_{25}}c_{2}(X_{25})\cdot H = 70
$$
 \end{Prop}
\begin{proof}
The basic strategy is to divide the construction of $X_{25}$ into two steps and repeat the similar argument in the previous subsection. 
The Grassmannian description guarantees the smoothness of $X_{25}$. 
It is also easy to see that $X_{25}$ is a Calabi--Yau threefold. 
$\int_{X_{25}}H^{3}$ and $\int_{X_{25}}c_{2}(X_{25})\cdot H$ may be determined in the same manner as before. 
The only non-trivial part is the determination of the Hodge numbers and we sketch a proof. \\

Let $X$ be the pfaffian sixfold associated to $(\mathscr{E}_{25},N_{1})$, which is isomorphic to Gr$(2,5)$. 
Then $Y=X_{25}$ is the pfaffian threefold associated to $(\mathscr{O}^{\oplus 5}_{X},N_{2})$. 
A straightforward computation with Lemma \ref{Res} shows that $h^{3}(Y,\mathscr{N}_{Y/X}^{\vee})=75$ and there is an exact sequence
\begin{align}
0\longrightarrow H^{2}(X,\varOmega_{X}\otimes_{\mathscr{O}_{\mathbb{P}_{\mathbf{w}}}}\mathscr{O}_{Y})
&\longrightarrow H^{3}(X,\mathscr{N}_{Y/X}^{\vee}\otimes_{\mathscr{O}_{\mathbb{P}_{\mathbf{w}}}}\mathscr{O}_{Y})\notag \\
\longrightarrow H^{3}(X,\varOmega_{\mathbb{P}^{9}}&\otimes_{\mathscr{O}_{\mathbb{P}_{\mathbf{w}}}}\mathscr{O}_{Y})
\longrightarrow H^{3}(X,\varOmega_{X}\otimes_{\mathscr{O}_{\mathbb{P}_{\mathbf{w}}}}\mathscr{O}_{Y})\longrightarrow 0. \notag
\end{align}
Combining this with the long exact sequence induced from the conormal sequence, we obtain 
\begin{align}
h^{2}(Y,\varOmega_{Y})
=&h^{3}(Y,\mathscr{N}_{Y/X}^{\vee})+h^{2}(X,\varOmega_{X}\otimes_{\mathscr{O}_{\mathbb{P}_{\mathbf{w}}}} \mathscr{O}_{Y})
-h^{3}(X,\varOmega_{X}\otimes_{\mathscr{O}_{\mathbb{P}_{\mathbf{w}}}} \mathscr{O}_{Y}) \notag\\
=&2 h^{3}(Y,\mathscr{N}_{Y/X}^{\vee})-h^{3}(X,\varOmega_{\mathbb{P}^{9}}\otimes_{\mathscr{O}_{\mathbb{P}_{\mathbf{w}}}} \mathscr{O}_{Y})=51.\notag
\end{align}
\end{proof}

The existence of a smooth Calabi--Yau threefold with the computed topological invariants was also predicted in \cite{van2}. 
This Calabi--Yau equation has two maximally unipotent monodromy points of the same type and this may be explained by the self-duality of  Gr$(2,5)$. 

\begin{Ex}
A complete intersection of a pfaffian variety associated to $\mathscr{F}_{10}=\StrO_{\mathbb{P}_{\mathbb{P}_{(1^{7},2)}}}^{\oplus5}$ 
and a quartic hypersurface in $\mathbb{P}_{(1^{7},2)}$ yields a smooth Calabi--Yau threefold $Y_{10}$ with the following topological invariants.
$$
h^{1,1}=1, \ h^{1,2}=101, \ \int_{Y_{10}}H^{3}=10, \ \int_{Y_{10}}c_{2}(Y_{10})\cdot H = 64
$$
We expect this to coincides with the double covering of Fano threefold in the list of C. Borcea \cite{van2}. 
\end{Ex}

\begin{Ex}
A complete intersection of a pfaffian variety associated to $\mathscr{F}_{5}=\StrO_{\mathbb{P}_{(1^{6},2,3)}}^{\oplus5}$ 
and a sextic hypersurface in $\mathbb{P}_{(1^{6},2,3)}$ yields a Calabi--Yau threefold $Y_{5}$.  
Assuming it is smooth, we can compute the topological invariants of $Y_{5}$.  
$$
h^{1,1}=1, \ h^{1,2}=156, \ \int_{Y_{5}}H^{3}=5, \ \int_{Y_{5}}c_{2}(Y_{5})\cdot H = 62
$$
Although we could not find a smooth example of $Y_{5}$, 
the existence of a Calabi--Yau threefold with the above invariants was predicted in \cite{van2}. 
\end{Ex}

The author is grateful to Makoto Miura for indicating the existence of $X_{25}$, $Y_{5}$ and $Y_{10}$.  
There are many choices for locally free sheaves $\mathscr{E}$ of odd rank and weights $ \mathbf{w}$ that yield Calabi--Yau threefolds, 
but there does not seem to exist any other smooth example that is not previously known. 
There is, nevertheless, an interesting example $X_{9}$, which we will analyze in Section 5.

\section{Mirror Symmetry for Degree 13 Pfaffian}
\subsection{Mirror Partner}
Our main aim of this subsection is to explicitly construct a mirror family of $X_{13}$. 
As $X_{13}$ is not a complete intersection Calabi--Yau threefold, the Batyrev--Borisov mirror construction is not applicable. 
We shall first briefly review the tropical mirror construction proposed by J. B\"{o}hm. 
His construction reproduces the conventional Batyrev--Borisov mirror construction for complete intersection Calabi--Yau in toric Fano varieties.  
For a thorough treatment of the tropical mirror construction, we refer the reader to the original paper \cite{boehm}. \\

We begin our exposition by recalling the Batyrev--Borisov mirror construction, using the standard notation in \cite{cox}. 
Let $M$ and $N=\mathrm{Hom}(M,\mathbb{Z})$ dual free abelian groups of rank $d$, 
and $M_{\mathbb{R}}$ and $N_{\mathbb{R}}$ be the scalar extension of $M$ and $N$ respectively.  
Suppose that $\mathbb{P}_{\Delta}$ is an $n$-dimensional toric variety associated with the normal fan $\Sigma_{\Delta}$ of an integral polytope $\Delta \subset M$. 
The Cox ring $S=\C[x_{r} | r \in \Sigma_{\Delta}(1)]$ of $\mathbb{P}_{\Delta}$ is graded 
by Chow group $A_{n-1}(\mathbb{P}_{\Delta})$ via the presentation sequence
$$
0 \longrightarrow M \stackrel{A}{\longrightarrow} \mathbb{Z}^{\Sigma_{\Delta}(1)} \longrightarrow A_{n-1}(\mathbb{P}_{\Delta}) \longrightarrow 0.
$$
Suppose that $\Delta$ is reflexive and given a nef-partition $\Delta=\Delta_{1}+\dots+\Delta_{k}$ or equivalently $\Sigma_{\Delta}(1)=I_{1}\cup \dots \cup I_{k}$, 
then a complete intersection Calabi--Yau variety $X=V(I) \subset \mathbb{P}_{\Delta}$ of dimension $d=n-k$ is the zero locus of a generic section
$(f_{i})_{i=1}^{k} \in H^{0}(\mathbb{P}_{\Delta}, \bigoplus_{i=1}^{k} \mathscr{O}_{\mathbb{P}_{\Delta}}(E_{i}))$, 
where $\bigotimes_{i=1}^{k} \mathscr{O}_{\mathbb{P}_{\Delta}}(E_{i}) \cong -K_{\mathbb{P}_{\Delta}}$ corresponds to the nef-partition. \\

Define $\nabla_{i}=\mathrm{Conv.}(\{0\}\cup I_{i})$ and the Minkowski sum $\nabla=\nabla_{1}+\dots+\nabla_{k} \subset N$. 
Then the following holds.
$$
\Delta^{*}=\mathrm{Conv.}(\nabla_{1}\cup \dots \cup \nabla_{k}), \ \ \ \nabla^{*}=\mathrm{Conv.}(\Delta_{1}\cup \dots \cup \Delta_{k})
$$
$\nabla=\nabla_{1}+\dots+\nabla_{k}$ is again reflexive and this gives a nef-partition of $\nabla$, called the dual nef-partition. 
We define a complete intersection Calabi--Yau variety $\check{X}$ by using $\nabla \subset N$.
Choosing a maximal projective subdivision of the normal fan of $\Delta$ and $\nabla$, 
we get families $\mathscr{X}$ and $\check{\mathscr{X}}$ of Calabi--Yau varieties, which are conjectured to form a mirror pair. 
It is important to observe that giving a nef-partition is essentially equivalent to determining a union of toric varieties $X_{0}=V(I_{0})$ 
to which a general fiber of the family $\mathscr{X}$ maximally degenerates.\\

Let $I_{0}$ be a reduced monomial ideal in the Cox ring $S$.  
The degree $0$ homomorphisms $\mathrm{Hom}(I_{0} , S/I_{0})_{0}$ form a finite dimensional vector space. 
The torus $T=\C^{\Sigma_{\Delta}(1)}$ acts on $S$ and thus on $\mathrm{Hom}(I_{0} , S/I_{0})_{0}$ as well. 
So the vector space has a basis of deformations which are characters of $T$. 
Any such character $\rho$ corresponds to an element $m_{\rho} \in M \cong \mathrm{Im}(A)$ as it is of degree $0$.  
Given a flat family $\mathscr{X}$ of Calabi--Yau varieties in $\mathbb{P}_{\Delta}$ with special fiber $X_{0}$ 
such that the corresponding ideal $I_{0} \subset S$ is a reduced monomial ideal. 
We represent the complex moduli space of a generic fiber $X$ of $\mathscr{X}$ by a one-parameter family $\mathscr{X}'$ as follows. 
Take a $T$-invariant basis $\rho_{1}, \dots \rho_{l} \in \mathrm{Hom}(I_{0} , S/I_{0})_{0}$ of the tangent space 
of the component of Hilbert scheme containing $\mathscr{X}$ at $X_{0}$ 
and assume that the tangent vector $v=\sum_{i}^{l}a_{i}\rho_{i}$ of $\mathscr{X}'$ at $X_{0}$ satisfies $a_{i}\ne 0$ for all $i$.  
The elements $\rho_{1}, \dots, \rho_{l}$ correspond to elements $m_{1},\dots, m_{l} \in M$ of the lattice of monomials of $\mathbb{P}_{\Delta}$. 
The construction of the first order deformation of a mirror family $\overline{\mathscr{X}'}$ comes with a natural map 
via the interpretation of lattice points as deformations and divisors (see also the monomial divisor map discussed in \cite{cox}).  
Take the convex hull $\nabla^{*}$ of $m_{1},\dots, m_{l}$ 
and define $\mathbb{P}_{\nabla}$ the toric variety associated to the normal fan of the (not necessarily integral) polytope $\nabla$. 
Then the toric divisors of $\mathbb{P}_{\nabla}$ and the induced divisors on a prospective mirror inside will correspond to deformations of $X_{0}$ in $\mathscr{X}$. 
The Bergman complex of $X_{0}$ defines a special fiber $\check{X}_{0} \subset\mathbb{P}_{\nabla}$ and 
the first order deformations $\overline{\check{\mathscr{X}}}$ contributing to the mirror degeneration $\check{X}_{0}$ 
are constructed by the lattice points of the support of $\mathrm{Strata}(X_{0})^{*} \subset  \Delta^{*}$. 
It is sufficient to know a given family up to first order deformation in case of complete intersection or pfaffian varieties as their deformations are unobstructed.\\ 

To relate $\overline{\check{\mathscr{X}}}$ to the initial family $\mathscr{X}$. 
We need to blow-down the ambient toric variety $\mathbb{P}_{\nabla}$ 
to obtain an orbifold quotient of a weighted projective space $\mathbb{P}_{\mathbf{w}}/G$, 
contracting divisors which do not correspond to Fermat deformation of $\mathscr{X}$ (see \cite{boehm} for the Fermat deformation).  
This blow-down is in general not unique and we choose appropriate one on case-by-case basis. 
The next one-parameter family was proposed as a mirror family of the degree 13 pfaffian Calabi--Yau threefold $X_{13}$. 
This family is obtained by deforming the special monomial fiber $\check{X}_{0}$ over $t=0$.   

\begin{Def}[J. B\"{o}hm \cite{boehm}]
Define $\check{\mathscr{X}}=\{\check{X}_{t}\}_{t\in \mathbb{P}^{1}}$ as the one-parameter flat family of the pfaffian Calabi--Yau threefolds 
associated to the following special skew-symmetric $5\times 5$ matrix $N_{t}$ parametrized by $t \in \mathbb{P}^{1}$.
$$
N_{t} = \begin{pmatrix}
                0                            & tx_{0}^{2}                          & x_{5}x_{6}      & x_{3}x_{4} & tx_{2}^{2} \\
                -tx_{0}^{2}       & 0                                             & t(x_{3}+x_{4})& x_{2}                    & x_{1} \\
                -x_{5}x_{6}  & -t(x_{3}+x_{4})             & 0                             & tx_{1}                  & x_{0} \\
                -x_{3}x_{4}  & -x_{2}                                     & -tx_{1}                    & 0                            & t(x_{5}+x_{6}) \\
                -tx_{2}^{2}       & -x_{1}                                     & -x_{0}                      & -t(x_{5}+x_{6}) & 0 \\
                \end{pmatrix} 
$$
 \end{Def}
 
The family $\check{\mathscr{X}}$ is nothing but a special one-parameter family of degree 13 pfaffian Calabi--Yau threefolds. 
More explicitly, the pfaffian ideal sheaf $\mathscr{I}_{\check{X}}$ of $\check{\mathscr{X}}$ is generated by
\begin{align}
 & P_{1} = x_{0}x_{2} - tx_{1}^{2} - t^{2}(x_{3} + x_{4})(x_{5} + x_{6}) \notag \\ 
 & P_{2} = x_{0}x_{3}x_{4} - tx_{5}x_{6}(x_{5} + x_{6}) -t^{2}x_{1}x_{2}^{2} \notag \\
 & P_{3} = x_{1}x_{3}x_{4} - tx_{2}^{3} - t^{2}x_{0}^{2}(x_{5} + x_{6}) \notag \\
 & P_{4} = x_{1}x_{5}x_{6} - tx_{0}^{3} - t^{2}x_{2}^{2}(x_{3} + x_{4}) \notag \\
 & P_{5} = x_{2}x_{5}x_{6} - tx_{3}x_{4} (x_{3} + x_{4}) -t^{2}x_{0}^{2}x_{1}.\notag 
 \end{align}
Since  $\check{X}_{t}$ is originally contained in the toric variety $\mathbb{P}^{6}/\mathbb{Z}_{13}$, 
$\mathbb{Z}_{13}$ acts on $\check{X}_{t}$ as 
\begin{align}
\zeta_{13} \cdot  [x_{0}:x_{1}:& \ x_{2}:x_{3}:x_{4}:x_{5}:x_{6}] = \notag \\
 & [x_{0}:\zeta_{13}^{4}x_{1}:\zeta_{13}^{8}x_{2}:\zeta_{13}^{10}x_{3}:\zeta_{13}^{10}x_{4}:\zeta_{13}^{11}x_{5}:\zeta_{13}^{11}x_{6}],\notag 
\end{align}
where $\zeta_{13} = e^{\frac{2\pi i}{13}}$. 
The fixed locus of the $\mathbb{Z}_{13}$-action consists of six points. 
Four of them $p_{i} = \{x_{i} \ne 0, x_{j}=0\ (j \ne i) \} \ (i=3,4,5,6)$ are singular and  
and other two $p_{i,i+1} = \{x_{i} + x_{i+1} = 0, \ x_{i} \ne 0 \ x_{j}=0\ (j \ne i,i+1)\} \ (i=3,5)$ are smooth.   

\begin{Prop}
For a generic choice of parameter $t \in \mathbb{P}^{1}$, the singular locus of $\check{X}_{t}$ consists of four points $p_{3}$, $p_{4}$, $p_{5}$ and $p_{6}$, 
each of which has multiplicity 12. 
\end{Prop}
\begin{proof}
Let us first work on the singular point $p_{3}$. 
In a neighborhood of $p_{3}$, since $P_{1,4,5}\neq 0$, $\check{X}_{t}$ is defined by the complete intersection of $P_{1}$, $P_{4}$ and $P_{5}$.
Then it is easily seen that the germ of this singularity is isomorphic to a compound Du Val singularity given by the equation 
$$
f(x,y,z,w) = x^{2} + y^{3} + z^{5} + zw^{2}=0, \ (x,y,z,w) \in \mathbb{C}^{4}.
$$
Here the action of $\mathbb{Z}_{13}$ is given by $\zeta_{13} \cdot (x,y,z,w) = (\zeta_{13}^{11}x, \zeta_{13}^{3}y, \zeta_{13}^{7}z, \zeta_{13} w)$.
The Milnor number of this singularity turns out to be 12. 
On the other hand, the Jacobian ideal of $\mathscr{I}_{\check{X}}$ has dimension 0 and degree 48 \footnote{This is done by Macaulay2 \cite{mac}.}. 
Due to symmetry, other singular points are of multiplicity 12 as well and hence we conclude the singular points are only $\{p_{i}\}_{i=3}^{6}$.
\end{proof}

Now we have a family of Calabi--Yau threefolds $\check{\mathscr{X}}=\{\check{X}_{t}\}_{t\in \mathbb{P}^{1}}$ parametrized by $t \in \mathbb{P}^{1}$. 
However, this is not an effective family because $\check{X}_{t} \cong \check{X}_{\zeta_{7} t}$ for $\zeta_{7} = e^{\frac{2\pi i}{7}} $ 
via the map 
$$
[x_{0}:x_{1}:x_{2}:x_{3}:x_{4}:x_{5}:x_{6}] \mapsto [x_{0}:\zeta_{7}^{3}x_{1}:x_{2},\zeta_{7}^{6}x_{3}:\zeta_{7}^{6}x_{4}:\zeta_{7}^{6}x_{5}:\zeta_{7}^{6}x_{6}].
$$

It is proved in \cite{sk} that the $cD_{4}$-singularity above does not admit any crepant resolution. 
However it turns out that the quotient $\{f(x,y,z,w)=0\}/\mathbb{Z}_{13}\subset \C^{4}/\mathbb{Z}_{13}$ admits a crepant resolution; 
in her Ph.D. thesis \cite{fa}, I. Fausk found a crepant resolution $\widetilde{\check{X}_{t}/\mathbb{Z}_{13}}$ of $\check{X}_{t}/\mathbb{Z}_{13}$ 
(for a generic choice of parameter $t\in \mathbb{P}^{1}$) 
and verified the relation 
$$
\chi(\widetilde{\check{X}_{t}/\mathbb{Z}_{13}})=120=-\chi(X_{13})
$$
as mirror symmetry predicts. 
The definition of the family $\check{\mathscr{X}}=\{\check{X}_{t}\}_{t\in \mathbb{P}^{1}}$ shall also be justified by calculating its Picard--Fuchs equation in the following subsection. 

\subsection{Period Map and Picard--Fuchs Equation}
Since $X_{13}$ is a smooth Calabi--Yau threefold, it has a nowhere vanishing holomorphic 3-form up to multiplication with a non-zero constant. 
Although a pfaffian variety is in general not a complete intersection and there is no way of explicitly getting one in general, 
there is an analogous way of obtaining a global section of $\omega_{X_{13}}\cong\varOmega_{X_{13}}^{3}$. 
For the sake of simplicity we restrict ourselves to the degree 13 pfaffian Calabi--Yau threefold $X_{13}$, but generalization to other pfaffian Calabi--Yau threefolds is straightforward. 

\begin{Prop}[E. R\o dland \cite{rod}]
Let $\sigma \in \mathfrak{S}_{5}$ be an element of the symmetric group of degree $5$.   
We have a nowhere vanishing global section of 
$\varOmega_{X_{13}}^{3} \cong \mathscr{O}_{X_{13}}(-7)\otimes_{\mathscr{O}_{X_{13}}} \bigwedge^{3} \mathscr{N}_{X_{13}/\mathbb{P}^{6}}$ 
given by
$$
\alpha = C_{\sigma} \mathrm{Res}_{X}\frac{P_{\sigma(1),\sigma(2),\sigma(3)}\Omega_{0}}{P_{\sigma(1)}P_{\sigma(2)}P_{\sigma(3)}},
$$
where $C_{\sigma}\in \C^{\times}$ is some constant and  
$$
\Omega_{0}=\frac{1}{(2\pi i)^{6}}\sum_{i=0}^{6}(-1)^{i}x_{i}dx_{0}\wedge \dots \wedge \widehat{dx_{i}} \wedge \dots \wedge dx_{6}.
$$
This expression is independent of the choice of $\sigma$ so long as the constant $C_{\sigma}$ is chosen appropriately.
\end{Prop}

\begin{proof}
First of all, the invariance of the integrand under scaling of the coordinates can be checked for each $\sigma$ 
and thus it is well-defined as a rational 6-form on $\mathbb{P}^{6}$. 
On the affine open set $U_{\sigma(4),\sigma(5)}=\{P_{\sigma(1),\sigma(2),\sigma(3)}\ne0\}$, $\{P_{\sigma(i)}\}_{i=1}^{3}$ forms a complete intersection 
and $P_{\sigma(1)}$, $P_{\sigma(2)}$ and $P_{\sigma(3)}$ can be seen as a part of the local coordinate. 
We may therefore assume that $\{P_{\sigma(1)}$, $P_{\sigma(2)}$, $P_{\sigma(3)}$, $x_{4}$, $x_{5}$, $x_{6}$, $x_{7}$ form the coordinate of $\mathbb{A}^{7}$, 
i.e. $\frac{\partial(P_{\sigma(1)},P_{\sigma(2)},P_{\sigma(3)})}{\partial(x_{1},x_{2},x_{3})}\ne0$. 
Since we have
$$
dP_{\sigma(1)}\wedge dP_{\sigma(2)}\wedge dP_{\sigma(3)}
=\sum_{i<j<k}\frac{\partial(P_{\sigma(1)},P_{\sigma(2)},P_{\sigma(3)})}{\partial(x_{i},x_{j},x_{k})}dx_{i}\wedge dx_{j}\wedge dx_{k},
$$
the residue theorem provides the following holomorphic 3-form on $X_{13} |_{U_{\sigma(4),\sigma(5)}}$.
$$
\alpha=C_{\sigma} \frac{P_{\sigma(1),\sigma(2),\sigma(3)}}
{(2\pi i)^{3} \frac{\partial(P_{\sigma(1)},P_{\sigma(2)},P_{\sigma(3)})}{\partial(x_{0},x_{1},x_{2})}}
\sum_{i=4}^{7}(-1)^{i}x_{i}dx_{4}\wedge \dots \wedge \widehat{dx_{i}} \wedge \dots \wedge dx_{7}
$$
On the other hand, $P_{\sigma(1),\sigma(2),\sigma(3)}$ vanishes if and only if $\{P_{\sigma(i)}\}_{i=1}^{3}$ does not form a complete intersection. 
Therefore the Jacobian $\frac{\partial(P_{\sigma(1)},P_{\sigma(2)},P_{\sigma(3)})}{\partial(x_{1},x_{2},x_{3})}$ divides $P_{\sigma(1),\sigma(2),\sigma(3)}$, and $\alpha$ is globally defined. 
Furthermore, the local description of $\alpha$ shows that it is nowhere vanishing on $X_{13}$. 
Since $X_{13}$ is Calabi--Yau threefold, $\varOmega_{X_{13}}^{3}$ is trivial and the expression of $\alpha$ for each $\nu$ is different merely by a constant. 
\end{proof}

Although a general member $\check{X}_{t}$ of the family is singular, a nowhere vanishing holomorphic 3-form $\alpha$ may be defined 
on the non-singular locus $\check{X}_{t} \setminus \mathrm{Sing}(\check{X}_{t})$. 
Then the period integral of the family $\check{\mathscr{X}}$ is defined as usual since integration can be performed on $3$-cycle away from the singular locus. 
For the sake of convenience we shall work with the singular threefold $\check{X}_{t}$ in $\mathbb{P}^{6}$ instead of a crepant resolution $\widetilde{\check{X}_{t}/\mathbb{Z}_{13}}$. 
Note that Picard--Fuchs equations invariant under resolution of singularities. 
It is also preserved under taking the quotient of the threefold by a finite group under which the 3-from $\alpha$ is invariant. 
More precisely, we can perform the integration on $\check{X}_{t}$ and obtain the genuine period integral by dividing by $13$. \\

At the origin $t=0$ the threefold $\check{X}_{t}$ decomposes into thirteen $3$-dimensional planes and 
hence the origin is a good candidate for a maximally unipotent monodromy point of the one-parameter family $\check{\mathscr{X}}$. 
The fundamental period integral $\Phi_{0}(t)$ (defined up to multiplication by a non-zero scalar) can be obtained by integrating a holomorphic 3-form on a torus cycle that vanishes at the origin $t=0$. 
In what follows we always assume that the fundamental period integral is normalized by setting $\Phi_{0}(0)=1$. 
Fix a 3-dimensional plane $H$ defined by $H=\{ x_{1}=x_{2}=x_{3}=0 \}$. 
On the domain $H\setminus (\{x_{4}=0\}\cup\{x_{5}=0\}\cup\{x_{6}=0\})$, 
there is a cycle given by $|\frac{x_{4}}{x_{0}}|=|\frac{x_{5}}{x_{0}}|=|\frac{x_{6}}{x_{0}}|=\epsilon$, 
which extends to a $3$-dimensional torus cycle $\gamma \in H_{3}(\check{X}_{t}, \mathbb{C})$ for $|t|\ll 1$. \\

\begin{Thm}
Let $\Phi_{0}(t) = \int_{\gamma}\alpha $ be the fundamental period integral of the one-parameter family $\{\check{X}_{t}\}_{t\in \mathbb{P}^{1}}$. 
Then $\Phi_{0}(t)$ has the following expansion near the origin $t=0$.   
$$
\Phi_{0}(t)=\sum_{n=0}^{\infty} \binom{2n}{n}^{2} \sum_{k=0}^{n} \binom{2n+k}{n} \binom{n}{k}^{2} t^{7n} 
$$
It can be observed that $\phi=t^{7} \in \mathbb{P}^{1}$ is the genuine moduli parameter of $\{\check{X}_{t}\}_{t\in \mathbb{P}^{1}}$.
We can thus write $\Phi_{0}(\phi)$ and $\{\check{X}_{\phi}\}_{\phi\in \mathbb{P}^{1}}$. 
Moreover, the Picard--Fuchs operator $\mathscr{D}$ of the family is
\begin{align}
 \mathscr{D} = & 13^{2}\Theta^{4}-\phi(59397\Theta^{4}+117546\Theta^{3}+86827\Theta^{2}+28054\Theta+3380) \notag \\
&+2^{4}\phi^{2}(6386\Theta^{4}-1774\Theta^{3}-17898\Theta^{2}-11596\Theta-2119) \notag \\ 
&+2^{8}\phi^{3}(67\Theta^{4}+1248\Theta^{3}+1091\Theta^{2}+312\Theta+26) -2^{12}\phi^{4}(2\Theta+1)^{4}, \notag
 \end{align}
where $\Theta$ is the Euler operator $\phi \frac{\partial}{\partial \phi}$.
\end{Thm}

\begin{proof}
Let us work on the affine open subset $U_{0}=\{x_{0}=1\}$. 
Fix a permutation, say $\nu = (2,4,5,1,3)$. 
For the sake of convenience we define $a_{i,j}$ to be
$$
(a_{i,j})= \begin{pmatrix}
                \frac{x_{5}^{2}x_{6}}{x_{3}x_{4}}t   &  \frac{x_{5}x_{6}^{2}}{x_{3}x_{4}}t  &  \frac{x_{1}x_{2}^{2}}{x_{3}x_{4}}t^{2} \\
                \frac{1}{x_{1}x_{5}x_{6}}t  &  \frac{x_{2}^{2}x_{3}}{x_{1}x_{5}x_{6}}t^{2}  &  \frac{x_{2}^{2}x_{4}}{x_{1}x_{5}x_{6}}t^{2}\\
                \frac{x_{3}^{2}x_{4}}{x_{2}x_{5}x_{6}}t  &  \frac{x_{3}x_{4}^{2}}{x_{2}x_{5}x_{6}}t  &  \frac{x_{1}}{x_{2}x_{5}x_{6}}t^{2}\\
                \end{pmatrix}.
$$
Then, near the origin, the period integral is described as
$$
\Phi_{0}(t)=\int_{\gamma}\mathrm{Res}_{X}\frac{P_{2,4,5}}{P_{2}P_{4}P_{5}} \Omega_{0}
=\int_{\Gamma}\frac{1}{\prod_{i=1,3,4}(1-\sum_{j=1}^{3}a_{i,j})} \cdot \bigwedge_{k=1}^{6} \frac{dx_{k}}{2 \pi i x_{k}},
$$
where $\Gamma=\{|x_{i}|=\epsilon \ (i=1,\dots 6)\}$. 
We then expand the denominator of the integrand as a power series in terms of $a_{i,j}$. 
The only terms that contribute the period integral is the products $\prod_{}a_{i,j}^{n_{i,j}}$ that is independent of $x_{i}$. 
Suppose $\prod_{}a_{i,j}^{n_{i,j}} \ (n_{i,j}\in \mathbb{Z}_{\ge0})$ does not contain any $x_{i}$, 
then $\prod_{}a_{i,j}^{n_{i,j}}$ is a product of
\begin{align}
& t_{1}=a_{1,1}a_{1,2}a_{2,3}a_{3,1}a_{3,3}=t^{7}=\phi \notag \\
& t_{2}=a_{1,1}a_{1,2}a_{2,2}a_{3,2}a_{3,3}=t^{7}=\phi \notag\\
& t_{3}=a_{1,1}a_{1,2}a_{1,3}a_{2,1}a_{3,1}a_{3,2}=t^{7}=\phi \notag
\end{align}
and it is easily checked that this expression is unique.
Therefore the period integral $\Phi_{0}(t)$ is essentially a function of $\phi=t^{7}$, and henceforth we write $\Phi_{0}(\phi)$. 
Note that this is compatible with the observation that $\check{X}_{t} \cong \check{X}_{\zeta_{7}t}$. 
Since 
$$
t_{1}^{a}t_{2}^{b}t_{3}^{c}=a_{1,1}^{a+b+c}a_{1,2}^{a+b+c}a_{1,3}^{c}a_{3,1}^{c}a_{3,2}^{b}a_{3,3}^{a}a_{4,1}^{a+c}a_{4,2}^{b+c}a_{4,3}^{a+c}
$$ and 
the coefficient of $\prod a_{i,j}^{n_{i,j}}$ appearing as an integrand of $\Phi_{0}(\phi)$ 
is given by $\prod_{i=1,3,4}\binom{n_{i,1}+n_{i,2}+n_{i,3}}{n_{i,1},n_{i,2},n_{i,3}}$, 
the period integral $\Phi_{0}(\phi)$ can be summarized as
\begin{align}
\Phi_{0}(\phi) = & \sum_{n=0}^{\infty}\sum_{a+b+c=n} \binom{2a+2b+3c}{a+b+c,a+b+c,c}\binom{a+b+c}{c,b,a}\binom{2a+2b+2c}{a+c,b+c,a+b}\phi^{n} \notag\\
= & \sum_{n=0}^{\infty} \sum_{k=0}^{n} \sum_{l=0}^{n-k} \binom{2n+k}{n} \binom{n+k}{n} \binom{n}{k} \binom{n-k}{l} \binom{2n}{n-l} \binom{n+l}{k+l} \phi^{n} \notag\\
= & \sum_{n=0}^{\infty} \sum_{k=0}^{n} \binom{2n+k}{n} \binom{n+k}{n} \binom{n}{k} \sum_{l=0}^{n-k} \binom{n-k}{l} \binom{2n}{n+k} \binom{n+k}{n-l} \phi^{n} \notag\\
= & \sum_{n=0}^{\infty} \binom{2n}{n} \sum_{k=0}^{n} \binom{n+k}{n} \binom{2n}{n}^{2} \sum_{l=0}^{n-k} \binom{n-k}{l} \binom{n+k}{n-l} \phi^{n} \notag\\
= & \sum_{n=0}^{\infty} \binom{2n}{n}^{2} \sum_{k=0}^{n} \binom{2n+k}{n} \binom{n}{k}^{2} \phi^{n},\notag
 \end{align}
where we used relations 
\begin{gather}
\binom{2n}{n-l}\binom{n+l}{k+l}=\binom{2n}{n+k}\binom{n+k}{n-l}, \ \  \binom{n+k}{n}\binom{2n}{n+k}=\binom{2n}{n}\binom{n}{k}, \notag \\
 \sum_{l=0}^{n-k}\binom{n-k}{l} \binom{n+k}{n-l} =\binom{2n}{k}.\notag 
\end{gather}
The period integral $\Phi_{0}(\phi)$ coincides with the power series solution of the Calabi--Yau equation of No.99 in \cite{van1}, 
whose Picard--Fuchs equation is exactly what we are looking for. 
\end{proof}

\begin{Cor}Let $\alpha_{1},\alpha_{2}$ be the roots of $256\phi^{2}+349\phi-1=0$.
Then Riemann's P-Scheme of $\mathscr{D}$ is given by the following.
$$
 \begin{Bmatrix}
 \begin{tabular}{c|ccccc} 
$\phi$  &  $0$  &  $\alpha_{1}$  &  $\alpha_{2}$   &  $13/16$  &  $\infty$ \\ \hline
 $\rho_{1}$               &  $0$  &  $0$  &  $0$  &  $0$  &  $1/2$  \\ \hline
 $\rho_{2}$               &  $0$  &  $1$  &  $1$  &  $1$  &  $1/2$  \\ \hline
 $\rho_{3}$               &  $0$  &  $1$  &  $1$  &  $3$  &  $1/2$  \\ \hline
 $\rho_{4}$              &  $0$  &  $2$  &  $2$  &  $4$  &  $1/2$ \notag
 \end{tabular}
 \end{Bmatrix}
$$
 The conifold points are $\alpha_{1}$ and $\alpha_{2}$.
\end{Cor}

$\mathscr{D}$ has a maximally unipotent monodromy point at the origin $\phi=0 \in \mathbb{P}^{1}$ as expected. 
Observe that $\infty$ is not a maximally unipotent monodromy point in the usual sense but very similar to that. 
This point will be further discussed later.

 \subsection{Picard--Fuchs Equation around $0$ and Curve Counting}
We now briefly review Gromov--Witten and BPS invariants. 
Let $X$ be a Calabi--Yau threefold. 
We define $N_{\beta}^{g}(X)=\int_{[\overline{M}_{g,0}(X,\beta)]^{vir}}1$ as the $0$-point genus $g$ Gromov--Witten invariant of $X$ in the curve class $\beta \in H_{2}(X,\mathbb{Z})$. 
Here $[\overline{M}_{g,0}(X,\beta)]^{vir}$ is the virtual fundamental class of the coarse moduli space of stable maps $\overline{M}_{g,0}(X,\beta)$ of 
expected complex dimension $(1-g)(\dim X-3) + \int_{\beta}c_{1}(X)=0$. 

\begin{Def}
Define BPS invariants $n_{\beta}^{g}(X)$ by the formula
$$
\sum_{\beta \ne0}\sum_{g\ge0}N_{\beta}^{g}(X)\lambda^{2g-2}q^{\beta}
=\sum_{\beta \ne0}\sum_{g\ge0}n_{\beta}^{g}(X)\sum_{k>0}\frac{1}{k}(2\sin(\frac{kt}{2}))^{2g-2}q^{k\beta}.
$$
LHS is the generating function of Gromov--Witten invariants of $N_{\beta}^{g}(X)$ of $X$ in all genera and all nonzero curve classes. 
Matching the coefficients of the two series yields equations determining $n_{\beta}^{g}(X)$ recursively in terms $N_{\beta}^{g}(X)$. 
\end{Def}

As the Picard-Fuchs operator $\mathscr{D}$ of $\check{\mathscr{X}}=\{\check{X}_{\phi}\}_{\phi\in \mathbb{P}^{1}}$ has a maximally unipotent monodromy point at $\phi=0$, 
we can define the mirror map $q(\phi)$ there and calculate the conjectural genus $g$ BPS invariants $n_{d}^{g} \ (d\in \N)$ of $X_{13}$. 
In what follows we will work on the case $g=0,1$ for simplicity. 
Since it is a routine work to calculate the mirror map and the Yukawa couplings, we omit the detail of those computations below. 
For a complete description, see for example \cite{COGP, cox}. \\

A good integral basis of $H_{3}(\check{X}_{\phi}, \mathbb{Z})$, 
which corresponds to the normalized solutions of $\mathscr{D}$ below, determines a canonical coordinate $q$ of the complexified K\"{a}hler moduli of $X_{13}$.
At the origin $\phi=0$ we have two normalized solutions of $\mathscr{D}$, $\Phi_{0}(\phi)$ and $\Phi_{1}(\phi)$. 
$\Phi_{0}(\phi)$ is the fundamental period integral normalized by setting $\Phi_{0}(0)=1$. 
The other period integral $\Phi_{1}(\phi)$ is of the form 
$$
\Phi_{1}(\phi) =(\log(\phi)) \Phi_{0}(\phi)+\Psi(\phi),
$$
where $\Psi(\phi)$ is regular at $\phi=0$ and $\Psi(0)=0$.
The mirror map is then defined by $q(\phi)=\exp(\frac{\Phi_{1}(\phi)}{\Phi_{0}(\phi)})$ and can be expanded as 
$$
q(\phi)=\phi+86\phi^{2}+12901\phi^{3}+2460318\phi^{4}+536898026\phi^{5}+ \dots
$$
Let us recall the definition of the quantum corrected Yukawa coupling 
$$
K_{ttt} (q) = \int_{X_{13}}H^{3}+(q\frac{d}{dq})^{3}\sum_{d\ge1}N_{d}(X_{13})q^{d} \in \mathbb{Q}[[q]]. 
$$
Using the mirror map $q(\phi)$, we may compute the quantum corrected Yukawa coupling
$$
K_{ttt} (q)= 13+647q+129975q^{2}+25451198q^{3}+5134100919q^{4}+ \dots. 
$$
We shall also apply the following BCOV formula \cite{BCOV} for genus $g=1$ Gromov--Witten potential $F_{1}(\phi)$ to $X_{13}$.
$$
F_{1}(\phi)=\frac{1}{2}\log \left\{\frac{\Phi_{0}(\phi)^{\frac{\chi(X_{13})}{12}-3-h^{1,1}}(q\frac{d\phi}{dq})}
{\mathrm{disc}(\phi)^{\frac{1}{6}}\phi^{\frac{\int_{X_{13}}c_{2}(X_{13})\cdot H}{12}+1}}\right\}
$$
Here we assumed that the exponent of the discriminant is $1/6$ as usual. 
The genus $g=0,1$ BPS invariants are given in the following table.   
\begin{center}
 \begin{tabular}{| c | l | l |} \hline
 $d$  &  $n_{d}^{0}$  &  $n_{d}^{1}$\\ \hline
 1      &  647 &  0  \\
 2      &  16166   &  0  \\
 3      &  942613   &  176\\
 4      &  80218296   &  164696\\
 5      &   8418215008   &  78309518\\
 \hline
 \end{tabular}
 \end{center}

Since we may write down explicit equations defining a degree 13 Calabi--Yau threefold $X_{13}$, we may in principle count the number of degree $d$ rational curves on general $X_{13}$ 
 and check that it coincides with $n_{d}^{0}$ as follows. 
Let us write a map $\mathbb{P}^{1} \rightarrow \mathbb{P}^{6}$ as 
$$
 [u:v] \mapsto [\sum_{i=0}^{d}a_{i}u^{i}v^{d-i}:\sum_{i=0}^{d}b_{i}u^{i}v^{d-i}:\dots:\sum_{i=0}^{d}g_{i}u^{i}v^{d-i}].
$$
Then the image of this map is contained in $X_{13}$ if and only if $P_{i}(\mathbf{x}(u,v))=0 \ (i=1,\dots, 5)$ for all $[u:v] \in \mathbb{P}^{1}$. 
This containment condition yields dependent equations in the variables $a_{0},a_{1},\dots,g_{d-1},g_{d}$. 
Since what we want to count is not maps from $\mathbb{P}^{1}$ to $X_{13}$ but rational curves in $X_{13}$, 
we must kill Aut$(\mathbb{P}^{1})$ by normalizing the map. 
As the proportional polynomials also define the same map, the number of the independent parameters turns out to be $7(d+1)-3-1$. 
We predict that the ideal generated by the dependent equations has dimension 0 and the degree $n_{d}^{0}$. \\

When $d=1$, we have nineteen dependent equations in ten variables.  
For a generic choice of $N$, we may suitably normalize the map and compute the degree of the ideal to get the answer $647$, 
as mirror symmetry predicts\footnote{This is done by Macaulay2 \cite{mac}.}. 
Explicit calculation is available upon request.

\subsection{Picard--Fuchs equation around $\infty$}

In his thesis \cite{rod}, E. R\o dland constructed a mirror family for the degree 14 pfaffian Calabi--Yau threefolds 
$X_{14}=\mathrm{Pfaff}(7) \cap \mathbb{P}^{6}$ by orbifolding the initial threefolds. 
His work is notable from two aspects. 
Firstly, this is the first example of mirror symmetry for a non-complete intersection Calabi--Yau threefold with $h^{1,1}=1$.
Secondly, the Picard--Fuchs equation of the mirror family $\check{X}_{14}$ has two maximally unipotent monodromy points; 
$\infty \in \mathbb{P}^{1}$ corresponds to the initial $X_{14}$ and  $0\in \mathbb{P}^{1}$ corresponds to $\mathrm{Gr}(2,7)\cap \mathbb{\check{P}}^{13} \subset \mathbb{\check{P}}^{20}$, 
which is the projective dual of $\mathrm{Pfaff}(7) \cap \mathbb{P}^{6} \subset \mathbb{P}^{20}$. 
In fact this pair is the first example of a derived equivalence between non-birational Calabi--Yau threefolds \cite{bc}. 
K. Hori and D. Tong presented how to describe these Calabi--Yau threefolds 
with GLSM using a non-abelian gauge group in two dimensions \cite{ht}. 
The link between the pfaffian $X_{14}$ and the Grassmannian sections $\mathrm{Gr}(2,7)_{1^{7}}$ was further studied in \cite{hos}, 
in which a thought-provoking phenomenon in the higher genus Gromov--Witten invariants is discovered.\\

In our case, $\infty$ is apparently an interesting point of $\mathscr{D}$ and it seems worthwhile to analyze it in detail
\footnote{This type of special point also appears when we consider, for instance, a Calabi--Yau threefold $\mathbb{P}^{7}_{2^{4}}$. 
The quantum differential equation of  $\mathbb{P}^{7}_{2^{4}}$ is $\theta^{4}-2^{4}q(2\theta+1)^{4}$.}. 
We first see that it makes sense to call $\infty$ a maximally unipotent monodromy point and calculate {\it virtual invariants} there. 
Changing the coordinate from $\phi$ to $1/\phi$ and transforming the gauge by $\sqrt{\phi}$ amount to the change, $\Theta \rightarrow -\Theta \rightarrow -\Theta-1/2,$
in the Euler operator. 
Let us also change the variable from $\phi$ to $-\phi$ for later use. 
Then the Picard--Fuchs operator becomes
\begin{align}
\mathscr{D}^{'}=&2^{20}\Theta^{4}-2^{8}\phi(1072\Theta^4-17824\Theta^3-10888\Theta^2-1976\Theta-145)\notag \\
&+2^{5}\phi^{2}(51088\Theta^4+116368\Theta^3-45264\Theta^2-14228\Theta-1397) \notag \\
&+13\phi^{3}(73104\Theta^4+1536\Theta^3-488\Theta^2+384\Theta+97) +13^{2}\phi^{4}(2\Theta+1)^4. \notag
\end{align}
 Although $\mathscr{D}^{'}$ has a maximally unipotent monodromy point at $\phi=0$, the integrality of mirror symmetry breaks.
 It is observed that there is a preferable choice of variable $\tilde{\phi}=\phi/2^{16}$, 
 with which the integrality of the normalized period, the mirror map and 
 {\it virtual BPS invariants} (see the next paragraph) still holds\footnote{S. Hosono pointed out that the change of the sign and the coefficient $1/2^{16}$ can be justified 
 by the analytic continuation of the local solutions about $0$ to $\infty$.}. 
 The Picard--Fuchs operator $\tilde{\mathscr{D}}$ with respect to this new variable is 
 \begin{align}
 \tilde{\mathscr{D}}=&
 \tilde{\Theta}^{4} -2^{4}\tilde{\phi}(1072\tilde{\Theta}^4-17824\tilde{\Theta}^3-10888\tilde{\Theta}^2-1976\tilde{\Theta}-145) \notag \\
 &+2^{17}\tilde{\phi}^{2}(51088\tilde{\Theta}^4+116368\tilde{\Theta}^3-45264\tilde{\Theta}^2-14228\tilde{\Theta}-1397)  \notag \\
 &+13\cdot2^{28}\tilde{\phi}^{3}(73104\tilde{\Theta}^4+1536\tilde{\Theta}^3-488\tilde{\Theta}^2+384\tilde{\Theta}+97) \notag \\
 &+ 13^{2}2^{44}\tilde{\phi}^{4}(2\tilde{\Theta}+1)^4. \notag
 \end{align}
 This is the Calabi--Yau equation of No.225 in \cite{van2}. 
 However, the positive Euler number corresponding to this equation \cite{van1} excludes a geometric interpretation by a Calabi--Yau threefold with $h^{1,1}=1$. \\
 
 Since the new operator $\tilde{\mathscr{D}}$ has a maximally unipotent monodromy point at the origin $\tilde{\phi}=0$, 
 it makes sense to speak of {\it virtual BPS invariants} $\tilde{n}_{d}^{0}\ (d\in \N)$ corresponding to the origin, 
\begin{center}
 \begin{tabular}{|c | l|} \hline
 $d$  &  $\tilde{n}_{d}^{0}$ \\ \hline
 1      &  70944$a$ \\
 2      &  107300032$a$ \\
 3      &  3707752060576$a$ \\
 4      &  66327758316665792$a$ \\
 5      &   1970671594871618215520$a$ \\
  \hline
 \end{tabular}
 \end{center}
where $a$ is supposed to be {\it the degree of virtual geometry} at the origin\footnote{$a$ is expected to be $1$ in \cite{van2}.}. 
We hope to understand the meaning of this sequence of numbers, which may not come from the conventional Calabi--Yau geometry. 
It would also be interesting to extend the Hori--Tong GLSM description \cite{ht} to our pfaffian Calabi--Yau threefolds.

\subsection{Conclusion}
It is classically known that the monodromy matrix of the quantum differential equation with respect to an appropriate basis 
is expressed in terms of the geometric invariants of the underlying Calabi--Yau threefold with one dimensional moduli. 
In what follows we assume that the origin is a maximally unipotent monodromy point. 
Then  $\int_{X}H^{3}$ and $\int_{X}c_{2}(X)\cdot H$ can be read off from the monodromy around the origin and the conifold point.  
After it is analytically continued to the origin, the conifold-period $z_{2}(t)$ has the form 
$$
z_{2}(t)=\frac{\int_{X}H^{3}}{6}t^{3}+\frac{\int_{X}c_{2}(X)\cdot H}{24}t+\frac{\int_{X}c_{3}(X)}{(2\pi i)^{3}}\zeta(3)
+\sum_{d=1}^{\infty}N_{d}^{0}(X)q^{d},
$$
where $q=e^{2\pi i t}$. 
So we obtain $\int_{X}c_{3}(X)$ as well and have consistency check of $\int_{X}c_{2}(X)\cdot H$. 
It was numerically verified in \cite{van2} that the invariants computed from the differential equation $\mathscr{D}$ 
coincides with the fundamental geometric invariants $\int_{X_{13}}H^{3}, \ \int_{X_{13}}c_{2}(X_{13})\cdot H$ and $\int_{X_{13}}c_{3}(X_{13})$.  
Our claim that $\widetilde{X_{t}/\mathbb{Z}_{13}}$ is a mirror partner of $X_{13}$ is based on the coincidence the fundamental geometric invariants mentioned above.
An alternative and preferable approach to the verification of mirror symmetry is direct computation of the Gromov--Witten invariants of $X_{13}$, such as \cite{tjo}. 
\begin{Conj}
The BPS invariants of the degree 13 pfaffian Calabi--Yau threefold $X_{13}$ 
coincides with the numbers $n_{d}^{g}\ (d\in \N)$ we calculated above, as mirror symmetry predicts.
\end{Conj}

\section{Mirror Symmetry for Degree 5, 7, 10 Pfaffians}
\subsection{Mirror Partners}
Inspired by the mirror symmetry for $X_{13}$, we will construct candidate mirror families of the Calabi--Yau threefolds we obtained in Section 2, except the degree 25 case. 
Since we do not know a systematic way of finding an appropriate family, we omit the finding procedure of the families in this paper. 
Some computation on singularities in this section are carried out with the aid of Macaulay 2. 
See also Appendix for the conjectural BPS invariants computed by the special families of Calabi--Yau threefolds in this section. \\

\begin{Def}
Define $\check{\mathscr{X}}_{5}=\{\check{X}_{5,t}\}_{t\in \mathbb{P}^{1}}$ as the one-parameter family of degree 5 pfaffian Calabi--Yau threefolds $\check{X}_{5,t}$ 
associated to the following special skew-symmetric $5\times 5$ matrix $N_{5,t}$ parametrized by $t \in \mathbb{P}^{1}$.
$$
N_{5,t}= \begin{pmatrix}
                0           & tx_{6}    & x_{4}     & x_{0}x_{1}    & tx_{5} \\
                -tx_{6}  & 0           & t(x_{0}^{2}+x_{1}^{2})    & x_{5}   & x_{2}x_{3}\\
                -x_{4}  & -t(x_{0}^{2}+x_{1}^{2})& 0              & t(x_{2}^{2}+x_{3}^{2})   & x_{6} \\
                -x_{0}x_{1}  & -x_{5} & -t(x_{2}^{2}+x_{3}^{2})   & 0           &  tx_{4}\\
                -tx_{5} & -x_{2}x_{3} & -x_{6}    & -tx_{4}  & 0 \\
                \end{pmatrix}
$$
\end{Def}

This is our candidate mirror family of the degree 5 pfaffian Calabi--Yau threefold $X_{5}$. 
Observe that the family degenerates to a union of toric varieties with normal crossings at the origin $t=0$.
In fact we choose $\check{X}_{5,0}$ as a candidate of the fiber over a maximally unipotent monodromy point and deform it so that the first order deformation resembles a Fermat variety. 
Then the deformation automatically extends to higher orders, so long as it is a pfaffian Calabi--Yau threefold. \\

$\mathbb{Z}_{10}$ acts on $\check{X}_{5,t}$ as
\begin{align}
\zeta_{10}\cdot [x_{0}:x_{1}:& \ x_{2}:x_{3}:x_{4}:x_{5}:x_{6}] = \notag \\
& [x_{0}:x_{1}:\zeta_{10}x_{2}:\zeta_{10}x_{3}:\zeta_{10}^{4}x_{4}:\zeta_{10}^{6}x_{5}:\zeta_{10}^{8}x_{6}], \notag 
\end{align}
where $\zeta_{10} = e^{\frac{2\pi i}{10}}$. 
There are four singular points $p_{i}^{\pm}= \{x_{i} \pm \sqrt{-1}x_{i+1} = 0, \ x_{i} \ne 0 \ x_{j}=0\ (j \ne i,i+1)\} \ (i=0,2)$ with multiplicity 2. 
Each singularity is locally isomorphic to 
$$
f_{5}(x,y,z,w)=x^{3}+y^{2}+zw=0, \ (x,y,z,w) \in \mathbb{C}^{4}, 
$$
where the action of $\mathbb{Z}_{10}$ is given by 
$\zeta_{10}\cdot (x, y, z, w) = (\zeta_{10}^{4}x, \zeta_{10}^{6}y, \zeta_{10}z, \zeta_{10}w)$. 
Since dim(Sing$(\check{X}_{5,t}))=0$ and deg(Sing$(\check{X}_{5,t}))=8$, there is no more singular point on $\check{X}_{5,t}$.\\

It is observed that $\check{\mathscr{X}}_{5}=\{\check{X}_{5,t}\}_{t\in \mathbb{P}^{1}}$ is not an effective family, 
as $\check{X}_{5,t} \cong \check{X}_{5,\zeta_{10} t}$ for $\zeta_{10} = e^{\frac{2\pi i}{10}} $ 
via the map 
$$
[x_{0}:x_{1}:x_{2}:x_{3}:x_{4}:x_{5}:x_{6}] \mapsto 
[x_{0}:x_{1}:x_{2}:\zeta_{10}^{5}x_{3}:\zeta_{10}^{4}x_{4}:\zeta_{10}^{7}x_{5}:\zeta_{10}^{9}x_{6}].
$$

\begin{Prop}
The fundamental period integral of the family $\{\check{X}_{5,t}\}_{t\in \mathbb{P}^{1}}$ around $t=0$ is given by 
$$
\Phi_{0}(\phi)=\sum_{n=0}^{\infty}\binom{2n}{n}\sum_{k}^{n}\binom{n}{k}\binom{n+k}{n}\binom{2n+2k}{n+k}\binom{2n+k}{2n-k}\phi^{n} 
$$
and the Picard-Fuchs operator $\mathscr{D}_{5}$ is given by 
\begin{align}
\mathscr{D}_{5}
= & \Theta^{4}-2^{2}\phi(500\Theta^4+976\Theta^3+677\Theta^2+189\Theta+19) \notag \\
& +2^{4}\phi^{2}(3968\Theta^4+3968\Theta^3-1336\Theta^2-1164\Theta-177) \notag \\
& -2^{10}\phi^{3}(500\Theta^4+24\Theta^3-37\Theta^2+6\Theta+3)+2^{12}\phi^{4}(2\Theta+1)^{4}, \notag
\end{align}
where we put $\phi=t^{10}$.
\end{Prop}
\begin{proof}
The computation is almost identical to the degree 13 pfaffian case.
\end{proof}
This Picard--Fuchs equation $\mathscr{D}_{5}$ is the Calabi--Yau equation of No.302 listed in \cite{van1}. 
The topological invariants computed from $\mathscr{D}_{5}$ coincide with those of $X_{5}$ as expected.  \\

\begin{Cor}
Let $\alpha_{1},\alpha_{2}$ be the roots of $ 256\phi^{2}-1968\phi+1=0$, then Riemann's P-Scheme of $\mathscr{D}_{5}$ is 
given by the following. 
$$
\begin{Bmatrix}
 \begin{tabular}{c|ccccc}
 $\phi$  &  $0$  &  $\alpha_{1}$  &  $\alpha_{2}$   &  $1/16$  &  $\infty$ \\ \hline
 $\rho_{1}$         &  $0$  &  $0$  &  $0$  &  $0$  &  $1/2$  \\ \hline
 $\rho_{2}$               &  $0$  &  $1$  &  $1$  &  $1$  &  $1/2$  \\ \hline
 $\rho_{3}$               &  $0$  &  $1$  &  $1$  &  $3$  &  $1/2$  \\ \hline
 $\rho_{4}$              &  $0$  &  $2$  &  $2$  &  $4$  &  $1/2$ 
 \end{tabular}
 \end{Bmatrix}
$$
\end{Cor}

Interestingly enough, the Picard--Fuchs operator $\mathscr{D}_{5}$ has two special points, namely $0$ and $\infty$. 
There is again a preferable new variable $\tilde{\phi}=\phi/2^{8}$ and 
the Picard--Fuchs equation around $\infty$ with respect to the new variable is identical to the initial one. 
Therefore it seems that both $0$ and $\infty$ correspond to the degree 5 Calabi--Yau threefold $X_{5}$ in this case. 

\begin{Def}
Define $\check{\mathscr{X}}_{7}=\{\check{X}_{7,t}\}_{t\in \mathbb{P}^{1}}$ as the one-parameter family of degree 7 pfaffian Calabi--Yau threefolds $\check{X}_{7,t}$ 
associated to the following special skew-symmetric $5\times 5$ matrix $N_{7,t}$ parametrized by $t \in \mathbb{P}^{1}$.
$$
N_{7,t}= \begin{pmatrix}
                0           & tx_{2}^{3}    & x_{0}x_{1}     & x_{5}    & tx_{6} \\
                -tx_{2}^{3}  & 0           & tx_{5}    & x_{6}   & x_{3}x_{4}\\
                -x_{0}x_{1}  & -tx_{5}& 0              & t(x_{3}+x_{4})   & x_{2} \\
                -x_{5}  & -x_{6} & -t(x_{3}+x_{4})   & 0           &  t(x_{0}+x_{1})\\
                -tx_{6} & -x_{3}x_{4} & -x_{2}    & -t(x_{0}+x_{1})  & 0 \\
                \end{pmatrix}
$$
\end{Def}

$\mathbb{Z}_{7}$ acts on $\check{X}_{7,t}$ as 
$$
\zeta_{7} \cdot [x_{0}:x_{1}:x_{2}:x_{3}:x_{4}:x_{5}:x_{6}] = 
[x_{0}:x_{1}:\zeta_{7}^{4}x_{2}:\zeta_{7}x_{3}:\zeta_{7}x_{4}:\zeta_{7}^{3}x_{5}:\zeta_{7}^{6}x_{6}],
$$
where $\zeta_{7} = e^{\frac{2\pi i}{7}}$. 
There are six fixed points, independent of the value of parameter $t$. 
We have dim(Sing$(\check{X}_{7,t}))=1$ and deg(Sing$(\check{X}_{7,t}))=4$. 
Sing$(\check{X}_{7,t})$ passes through two of the above fixed points, namely $p_{i,i+1} = \{x_{i} + x_{i+1} = 0, \ x_{i} \ne 0 \ x_{j}=0\ (j \ne i,i+1)\} \ (i=0,3)$.\\

It is observed that $\check{\mathscr{X}}_{7}=\{\check{X}_{7,t}\}_{t\in \mathbb{P}^{1}}$ is not an effective family, 
as $\check{X}_{7,t} \cong \check{X}_{7,\zeta_{9} t}$ for $\zeta_{9} = e^{\frac{2\pi i}{9}} $ 
via the map 
$$
[x_{0}:x_{1}:x_{2}:x_{3}:x_{4}:x_{5}:x_{6}] \mapsto [x_{0}:x_{1}:\zeta_{9}^{2}x_{2}:x_{3}:x_{4}:\zeta_{9}^{8}x_{5}:\zeta_{9}^{8}x_{6}].
$$ 

\begin{Prop}
The fundamental period integral of the family $\{\check{X}_{7,t}\}_{t\in \mathbb{P}^{1}}$ around $t=0$ is given by 
$$
\Phi_{0}(\phi)=\sum_{n=0}^{\infty}\binom{2n}{n}\sum_{k=0}^{2n}\binom{n+k}{k}\binom{2n}{k}^{2}\phi^{n}
$$
and the Picard--Fuchs operator $\mathscr{D}_{7}$ is given by 
\begin{align}
\mathscr{D}_{7}
= & 7^{2}\Theta^{4}-2\cdot3\cdot7\phi(1272\Theta^4+2508\Theta^3+1779\Theta^2+525\Theta+56) \notag \\
& +2^{2}3\phi^{2}(43704\Theta^4+38088\Theta^3-25757\Theta^2-20608\Theta-3360) \notag \\
& -2^{4}3^{3}\phi^{3}(2736\Theta^4-1512\Theta^3-1672\Theta^2-357\Theta-14)  \notag \\
& -2^{6}3^{5}\phi^{4}(2\Theta+1)^2(3\Theta+1)(3\Theta+2), \notag
\end{align}
where we put $\phi=t^{9}$. 
\end{Prop}
\begin{proof}
The computation is almost identical to the degree 13 pfaffian case.
\end{proof}

This Picard--Fuchs equation $\mathscr{D}_{7}$ is the Calabi--Yau equation of No.109 listed in \cite{van1}. 
The topological invariants computed from $\mathscr{D}_{7}$ coincide with those of $X_{7}$ as expected. \\

\begin{Cor}Let $\alpha_{1},\alpha_{2}$ be the roots of $ 432\phi^{2}+1080\phi-1=0$, then Riemann's P-Scheme of $\mathscr{D}_{7}$ is 
given by the following. 
$$
\begin{Bmatrix}
 \begin{tabular}{c|ccccc}
 $\phi$  &  $0$  &  $\alpha_{1}$  &  $\alpha_{2}$   &  $7/36$  &  $\infty$ \\ \hline
 $\rho_{1}$         &  $0$  &  $0$  &  $0$  &  $0$  &  $1/2$  \\ \hline
 $\rho_{2}$               &  $0$  &  $1$  &  $1$  &  $1$  &  $1/2$  \\ \hline
 $\rho_{3}$               &  $0$  &  $1$  &  $1$  &  $3$  &  $1/3$  \\ \hline
 $\rho_{4}$              &  $0$  &  $2$  &  $2$  &  $4$  &  $2/3$ 
 \end{tabular}
 \end{Bmatrix}
$$
$0$ is the only maximally unipotent monodromy point of $\mathscr{D}_{7}$.
\end{Cor}

\begin{Def}
Define $\check{\mathscr{X}}_{10}=\{\check{X}_{10,t}\}_{t\in \mathbb{P}^{1}}$ as the one-parameter family of degree 10 pfaffian Calabi--Yau threefolds $\check{X}_{10,t}$ 
associated to the following special skew-symmetric $5\times 5$ matrix $N_{10,t}$ parametrized by $t \in \mathbb{P}^{1}$. 
$$
N_{10,t}= \begin{pmatrix}
                0           & tx_{4}^{2}   & x_{0}x_{1}     & x_{6}    & t(x_{2}+x_{3}) \\
                -tx_{4}^{2}  & 0           & tx_{6}    & x_{2}x_{3}   & x_{5}\\
                -x_{0}x_{1}  & -tx_{6}& 0              & tx_{5}^{2}   & x_{4} \\
                -x_{6}  & -x_{2}x_{3} & -tx_{5}^{2}   & 0           &  t(x_{0}+x_{1})\\
                -t(x_{2}+x_{3}) & -x_{5} & -x_{4}    & -t(x_{0}+x_{1})  & 0 \\
                \end{pmatrix}
$$
\end{Def}

$\mathbb{Z}_{10}$ acts on $\check{X}_{10,t}$ as
\begin{align}
\zeta_{10} \cdot [x_{0}:x_{1}:& \ x_{2}:x_{3}:x_{4}:x_{5}:x_{6}] = \notag \\
& [x_{0}:x_{1}:\zeta_{10}^{6}x_{2}:\zeta_{10}^{6}x_{3}:\zeta_{10}^{9}x_{4}:\zeta_{10}^{7}x_{5}:\zeta_{10}^{1}x_{6}],\notag 
\end{align}
where $\zeta_{10} = e^{\frac{2\pi i}{10}}$. 
There are six singular points under the $\mathbb{Z}_{10}$-action. 
Four of them $p_{i} = \{x_{i} \ne 0, x_{j}=0\ (j \ne i) \} \ (i=0,1,2,3)$ appear with multiplicity 12 
and other two $p_{i,i+1} = \{x_{i} + x_{i+1} = 0, \ x_{i} \ne 0 \ x_{j}=0\ (j \ne i,i+1)\} \ (i=0,2)$ with multiplicity 7. 
The singularity at $p_{i}$ is locally isomorphic to 
$$
 f_{10}(x,y,z,w)=x^{4}+y^{2}+z^{2}w+zw^{2}=0 \ (x,y,z,w) \in \mathbb{C}^{4}, 
$$
where the action of $\mathbb{Z}_{10}$ is given by $\zeta_{10}\cdot (x, y, z, w) = (\zeta_{10}^{7}x, \zeta_{10}^{9}y, \zeta_{10}^{6}z, \zeta_{10}^{6}w)$. 
The singularity at $p_{i,i+1}$ is locally isomorphic to 
$$
g_{10}(x,y,z,w)= x^{8}+y^{2}+zw=0,
$$
where $\zeta_{10}\cdot (x, y, z, w) = (\zeta_{10}^{9}x, \zeta_{10}y, \zeta_{10}^{6}z, \zeta_{10}^{6}w)$. 
Since dim(Sing$(\check{X}_{10}))=0$ and deg(Sing$(\check{X}_{10}))=62$, there is no more singular point on $\check{X}_{10,t}$.\\

It is observed that $\check{\mathscr{X}}_{10}=\{\check{X}_{10,t}\}_{t\in \mathbb{P}^{1}}$ is not an effective family, 
as $\check{X}_{10,t} \cong X_{10,\zeta_{16}^{2} t}$ for $\zeta_{16} = e^{\frac{2\pi i}{16}} $ 
via the map 
$$
[x_{0}:x_{1}:x_{2}:x_{3}:x_{4}:x_{5}:x_{6}] \mapsto [x_{0}:x_{1}:x_{2}:x_{3}:\zeta_{16}^{3}x_{4}:\zeta_{16}^{3}x_{5}:\zeta_{16}^{7}x_{6}].
$$

\begin{Prop}
The fundamental period integral of the family $\{\check{X}_{10,t}\}_{t\in \mathbb{P}^{1}}$ around $t=0$ is given by 
$$
\Phi_{0}(\phi)= \sum_{n=0}^{\infty}\binom{2n}{n}\sum_{k=0}^{2n}(-1)^{k+n}\binom{2n}{k}^{4}\phi^{n}
$$
and the Picard--Fuchs operator $\mathscr{D}_{10}$ is given by 
\begin{align}
\mathscr{D}_{10} 
= & 5^{2}\Theta^4-2^{2}5\phi(688\Theta^4+1352\Theta^3+981\Theta^2+305\Theta+35) \notag \\
& +2^{4}\phi^{2}(5856\Theta^4+7008\Theta^3+96\Theta^2-1260\Theta-265) \notag \\
& -2^{10}\phi^{3}(176\Theta^4+120\Theta^3+69\Theta^2+30\Theta+5)+2^{12}\phi^{4}(2\Theta+1)^4, \notag
\end{align}
where we put $\phi=t^{8}$. 
\end{Prop}
\begin{proof}
The computation is almost identical to the degree 13 pfaffian case.
\end{proof}

This Picard-Fuchs equation $\mathscr{D}_{10}$ is the Calabi--Yau equation of No.263 listed in \cite{van1}.  
The topological invariants computed from $\mathscr{D}_{10}$ coincide with those of $X_{10}$ as expected. \\

\begin{Cor}Let $\alpha_{1},\alpha_{2}$ be the roots of $ 256\phi^{2}-544\phi+1=0$, then Riemann's P-Scheme of $\mathscr{D}_{10}$ is
$$
\begin{Bmatrix}
 \begin{tabular}{c|ccccc}
 $\phi$  &  $0$  &  $\alpha_{1}$  &  $\alpha_{2}$   &  $5/16$  &  $\infty$ \\ \hline
 $\rho_{1}$               &  $0$  &  $0$  &  $0$  &  $0$  &  $1/2$  \\ \hline
 $\rho_{2}$               &  $0$  &  $1$  &  $1$  &  $1$  &  $1/2$  \\ \hline
 $\rho_{3}$               &  $0$  &  $1$  &  $1$  &  $3$  &  $1/2$  \\ \hline
 $\rho_{4}$              &  $0$  &  $2$  &  $2$  &  $4$  &  $1/2$ \\
 \end{tabular}
 \end{Bmatrix}.
$$
\end{Cor}

The Picard--Fuchs operator $\mathscr{D}_{10}$ has two special points $0$ and $\infty$. 
The Picard--Fuchs operator $\tilde{\mathscr{D}}_{10}$ around $\infty$ with respect to the new variable $\tilde{\phi}=1/(\phi2^{12})$ is
\begin{align}
\tilde{\mathscr{D}}_{10}
=&\tilde{\Theta}^4-2^{4}\tilde{\phi}(704\tilde{\Theta}^4+928\tilde{\Theta}^3+612\tilde{\Theta}^2+148\tilde{\Theta}+13)\notag \\
&+2^{12}\tilde{\phi}^{2}(5856\tilde{\Theta}^4+4704\tilde{\Theta}^3-1632\tilde{\Theta}^2-972\tilde{\Theta}-121)\notag \\
&-2^{20}5\tilde{\phi}^{3}(2752\tilde{\Theta}^4+96\tilde{\Theta}^3-60\tilde{\Theta}^2+24\tilde{\Theta}+7)
+2^{28}5^{2}\tilde{\phi}^{4}(2\tilde{\Theta}+1)^{4}. \notag
\end{align}
This is the Calabi--Yau equation of No.271 listed in \cite{van1}. 
It is unknown whether or not there exists a Calabi--Yau threefold with topological invariants predicted in \cite{van2}. \\

A general member of the one-parameter families constructed in this section is quite singular just as the degree 13 case. 
It is still unsettled whether or not a general member admits any crepant resolution. 
Hence our verification of mirror phenomena is again based on the monodromy calculation of the Picard--Fuchs equation of the our special one-parameter family \cite{van2}. 
\begin{Conj}
The BPS invariants of the pfaffian Calabi--Yau threefold $X_{i} \ (i=5,7,10)$ coincides with the numbers $n_{d}^{g} \ (d\in \N)$ listed in Appendix as mirror symmetry predicts.
\end{Conj}

\section{Another Example}
Although we could not find any more (new) smooth pfaffian Calabi--Yau threefolds in weighted projective spaces, 
there is an interesting example $X_{9}$ defined as follows. 
\begin{Def}
Define a degree $9$ Calabi--Yau threefold $X_{9} \subset \mathbb{P}_{(1^{6},2)}$ as a pfaffian variety associated to the locally free sheaf 
$\mathscr{E}_{9}=\mathscr{O}_{\mathbb{P}_{(1^{6},2)}}(2)\oplus \mathscr{O}_{\mathbb{P}_{(1^{6},2)}}(1)^{\oplus2}\oplus \mathscr{O}_{\mathbb{P}_{(1^{6},2)}}^{\oplus2}$.
 \end{Def}
 
$X_{9}$ turns out to be isomorphic to a complete intersection Calabi--Yau threefold $\mathbb{P}^{5}_{3^{2}}$. 
Therefore $X_{9}$ admits a twofold interpretation.  
If we regard $X_{9}$ as a pfaffian Calabi--Yau threefold, we can apply to it the orbifold mirror construction we studied in the preceding sections. 

\begin{Def}
Define $\check{\mathscr{X}}_{9}=\{\check{X}_{9,t}\}_{t\in \mathbb{P}^{1}}$ as the one-parameter family of degree 9 pfaffian Calabi--Yau threefolds $\check{X}_{9,t}$ 
associated to the following special skew-symmetric $5\times 5$ matrix $N_{9,t}$ parametrized by $t \in \mathbb{P}^{1}$.
$$
 N_{9,t}= \begin{pmatrix}
                0           & x_{0}x_{1}x_{2}          & 0     & tx_{6}    & x_{3}x_{4} \\
                -x_{0}x_{1}x_{2}  & 0                            & x_{6}    & t(x_{3}+x_{4})   & tx_{5}\\
                0                     & -x_{6}                  & 0              & x_{5}   & t(x_{0}+x_{1}+x_{2}) \\
                -tx_{6}           & -t(x_{3}+x_{4})   & -x_{5}   & 0           &  1\\
                x_{3}x_{4}    & -tx_{5}                  & -(x_{0}+x_{1}+x_{2})    & -1  & 0 \\
                \end{pmatrix}
 $$
 \end{Def}
$\check{X}_{9,t}$ is actually a complete intersection Calabi--Yau threefold defined by the quadric $P_{0}$ and the two cubics $P_{1}$ and $P_{2}$. 
This one-parameter family $\check{\mathscr{X}}_{9}$ is not isomorphic to the conventional mirror family of $\mathbb{P}^{5}_{3^{2}}$ defined by
 \begin{gather}
 x_{0}x_{1}x_{2}+t(x_{3}^{3}+x_{4}^{3}+x_{5}^{3}) \notag\\
 x_{3}x_{4}x_{5}+t(x_{0}^{3}+x_{1}^{3}+x_{2}^{3}). \notag
 \end{gather}
A general member of this family is a smooth Calabi--Yau threefold, while a general member of $\check{\mathscr{X}}_{9}$ is singular along a curve. 
It is, however, observed that the two families share the same normalized period integral and Picard--Fuchs operator 
$$
\Phi_{0}(\phi)= \sum_{n=0}^{\infty}\binom{3n}{n}^{2}\binom{2n}{n}^{2}\phi^{n}, \ \ \
\mathscr{D}_{9}=\Theta^{4}-3^{2}\phi(3\Theta+1)^2(3\Theta+2)^2, 
$$
where we put $\phi=t^{8}$. 
These two families may bridge the two mirror constructions. 
\\

It is classically known that a mirror family for a given family Calabi--Yau threefolds can be constructed by taking special loci of the initial family, 
which are not necessarily on the Fermat points emphasized by the initial construction inspired by the conformal field theories. 
For more details we refer the reader to \cite{dgj} and the reference therein.\\

\newpage

\section*{Appendix}
\ \\

\begin{center}
 \begin{tabular}{| c | l | l |}
 \hline
 \multicolumn{3}{| c |}{$X_{5}$}\\
  \hline
 $d$  &  $n_{d}^{0}$ &  $ n^{1}_{d}$  \\ \hline
 1      &   2220                                                                    &   0\\
 2      &   285520                                                               &   460\\
 3      &    95254820                                                         &   873240\\
 4      &   47164553340                                                  &   1498922677\\
 5      &   28906372957040                                          &   2306959237408\\
\hline
 \end{tabular}
  \end{center}
  
  \ \\
 \\

 \begin{center}
 \begin{tabular}{| c | l | l |} 
 \hline
 \multicolumn{3}{| c |}{$X_{7}$}\\
 \hline
 $d$  &  $n_{d}^{0}$  &  $n_{d}^{1}$ \\ \hline
 1      &   1434   &   0 \\
 2      &   103026   &   26  \\
 3      &   18676572  &    53076  \\
 4      &   4988009280  &   65171063  \\
 5      &   1646787631350     &   63899034076  \\
 \hline
 \end{tabular}
  \end{center}
 \ \\
 \\
 
  \begin{center}
 \begin{tabular}{| c | l | l |} 
 \hline
 \multicolumn{3}{| c |}{$X_{10}$}\\
 \hline
 $d$  &  $n_{d}^{0}$ &  $n_{d}^{1}$\\ \hline
 1      &   888                                                      &  0 \\
 2      &  33084                                                  &  1\\
 3      &  3003816                                             &   2496\\
 4      &  399931068                                        &   2089393\\
 5      &   65736977760                                  &   1210006912\\
  \hline \hline
 
 $d$  &  $\tilde{n}_{d}^{0}$ &  $\tilde{n}_{d}^{1}$\\ \hline
 1      &  2400$a$        & 40\\
 2      &  1829880$a$         & 138040 \\                             
 3      &  2956977632$a$     &  687719624\\
 4      &  7117422755016$a$   &  3822563543952\\
 5       &  21319886408804640$a$    &  21893828822263288\\
  \hline
 \end{tabular}
 \end{center}
 \footnote{$a$ is expected to be $2$ in \cite{van2}.}

\par\noindent{\scshape \small
Department of Mathematics, University of British Columbia \\
51984 Mathematics Rd , Vancouver, BC, V6T 1Z2, CANADA.}
\par\noindent{\ttfamily kanazawa@math.ubc.ca}

\end{document}